\documentclass[amssymb,amsfonts,refcheck,12pt,verbatim,righttag]{amsart}

\usepackage{amsmath}
\usepackage{amsfonts}
\usepackage{amssymb}
\usepackage{mathrsfs}
\usepackage{savesym}

\usepackage{graphicx}

\usepackage{graphics,color}

\usepackage{amsthm}
\newtheorem{theorem}{Theorem}[section]
\newtheorem{lem}[theorem]{Lemma}
\newtheorem{proposition}[theorem]{Proposition}
\newtheorem{corollary}[theorem]{Corollary}
\newtheorem{conjecture}[theorem]{Conjecture}
\newtheorem{problem}[theorem]{Problem}
\newtheorem{example}[theorem]{Example}
\theoremstyle{definition}
\newtheorem{definition}[theorem]{Definition}
\newtheorem{algorithm}[theorem]{Algorithm}

\theoremstyle{remark}
\newtheorem{remark}[theorem]{Remark}

\numberwithin{equation}{section}

\usepackage[colorlinks,citecolor=blue,pagebackref,urlcolor=black,hypertexnames=false]{hyperref}

\title[Estimates on the dimension of self-similar measures]{Estimates on the dimension of self-similar measures with overlaps}

\author{De-Jun FENG}
\address{Department of Mathematics\\
The Chinese University of Hong Kong\\
Shatin,  Hong Kong\\}
\curraddr{}
\email{\href{mailto:djfeng@math.cuhk.edu.hk}{djfeng@math.cuhk.edu.hk}}
\thanks{}

\author{Zhou Feng}
\address{Department of Mathematics\\
The Chinese University of Hong Kong\\
Shatin,  Hong Kong\\}
\curraddr{}
\email{\href{mailto:zfeng@math.cuhk.edu.hk}{zfeng@math.cuhk.edu.hk}}
\thanks{}



\date{}

\newcommand{\calB}{\mathcal{B}}
\newcommand{\calP}{\mathcal{P}}

\newcommand{\N}{{\Bbb N}}
\newcommand{\R}{{\Bbb R}}

\newcommand{\huaD}{\mathscr{D}}

\newcommand{\huaA}{\mathscr{A}}
\newcommand{\huaE}{\mathscr{E}}
\newcommand{\huaC}{\mathscr{C}}

\newcommand{\bfR}{\mathbb{R}}

\newcommand{\bfN}{\mathbb{N}}

\newcommand{\erfenzi}{\frac{1}{2}}
\newcommand{\Erfenzi}{\dfrac{1}{2}}

\newcommand{\Union}[2]{\overset{#2}{\underset{#1}{\bigcup}}}

\newcommand{\unitinterval}{[0,1]}
\DeclareMathOperator{\diam}{diam}
\newcommand{\mubeta}{\mu_{\beta}}

\newcommand{\dimmubeta}{\dim(\mu_{\beta})}
\newcommand{\dimmu}[1]{\dim_{H}(\mu_{#1})}

\newcommand{\eucild}[1]{\mathbb{R}^{d}}

\newcommand{\conditionE}{H_{m}(\calP\vert\pi^{-1}\calB(\eucild{d}))}
\newcommand{\conditionD}[1]{H_{m}(\calP\vert\pi^{-1}\huaD_{#1})}
\newcommand{\conditionA}[1]{H_{m}(\calP\vert\pi^{-1}\huaA_{#1})}
\newcommand{\conditionhuaE}[1]{H_{m}(\calP\vert\pi^{-1}\huaE_{#1})}
\newcommand{\nliminf}{\lim_{n\to\infty}}
\newcommand{\nlimsup}{\underset{n\to\infty}{\limsup}}
\newcommand{\sigalgebra}{\mbox{$\sigma$-algebra }}
\newcommand{\subsigalgebra}{\mbox{sub-$\sigma$-algebra }}

\newcommand{\mubetanterval}[2]{\mu_{\beta}([#1,#2])}
\newcommand{\mubetaprimenterval}[2]{\mu_{\beta'}\left([#1,#2]\right)}
\newcommand{\abs}[1]{\left\lvert #1 \right\rvert}
\newcommand{\Braced}[1]{\left \lbrace #1 \right \rbrace }
\newcommand{\inversebeta}{\dfrac{1}{\beta}}
\newcommand{\inversebetaprime}{\dfrac{1}{\beta'}}
\newcommand{\onetwointerval}{(1,2)}
\newcommand{\codingmap}[2]{S_{\beta^{#1},\varepsilon_{1}}\circ \cdots \circ S_{\beta^{#1},\varepsilon_{n}}\left(#2\right)    }
\newcommand{\distance}{\abs{\beta-\beta'}}
\newcommand{\optest}{\dfrac{\distance}{\beta\beta'\left[\min(\beta,\beta')-1\right]}}
\newcommand{\gfun}[1]{g_{\varepsilon}(#1) }
\newcommand{\gprimfun}[1]{g_{\varepsilon}'(#1) }
\newcommand{\normed}[1]{\left \lVert #1 \right \rVert}

\begin{document}
\begin{abstract}
	In this paper, we provide an algorithm to estimate from below the dimension of self-similar measures with overlaps. As an application, we show that for any $ \beta\in(1,2) $, the dimension of the Bernoulli convolution $ \mubeta $ satisfies
	\[ \dimmubeta \geq 0.98040856,\]
	which improves a previous uniform lower bound $0.82$ obtained by Hare and Sidorov \cite{HareSidorov2018}.
This new uniform lower bound is very close to the known numerical approximation $ 0.98040931953\pm 10^{-11}$ for $\dim \mu_{\beta_3}$, where $ \beta_{3} \approx 1.839286755214161$ is the largest root of the polynomial $ x^{3}-x^{2}-x-1$.
	Moreover, the infimum $\inf_{\beta\in (1,2)}\dimmubeta  $ is attained at a parameter $\beta_*$ in a small interval
	\[ (\beta_{3} -10^{-8},  \beta_{3} + 10^{-8}).\]
	 When $\beta$ is a Pisot number,   we express $\dimmubeta$ in terms of the measure-theoretic entropy of the equilibrium measure for certain matrix pressure function, and present an algorithm to estimate $\dim (\mu_\beta)$ from above as well.
\end{abstract}

\maketitle

\section{Introduction}
\label{S-1}

This paper is devoted to the dimension estimations of self-similar measures with overlaps.

Let us first introduce some notation and definitions. By an iterated function system (IFS) on $\R^d$ we mean a finite family   $\{S_i\}_{i=1}^\ell$ of  contracting transformations on $ \R^d$.  By Hutchinson \cite{Hutchinson1981}, for a given IFS $\{S_i\}_{i=1}^\ell$ on $\R^d$ there is a unique nonempty compact set $ K \subset \eucild{d}  $ such that
\[ K= \Union{i=1}{\ell} \: S_{i}(K).\]
The set $ K $ is called the {\it attractor} of   $\{S_i\}_{i=1}^\ell$. Moreover, $K$ is called a {\it self-similar set} if all $ S_{i} $ are similarities, and a {\it self-affine set} if all $ S_{i} $ are affine maps.

Let $(p_{1},\ldots, p_{\ell}) $ be a probability vector, that is, $ p_{i}>0 $ for all $i$ and   $ \sum_{i=1}^\ell p_{i} =1 $. It is well-known  \cite{Hutchinson1981} that there is a unique Borel probability measure $\mu$ on $ \R^d$ such that
  \[ \mu =\sum_{i=1}^{\ell} p_{i}\mu\circ S_{i}^{-1}. \]
Moreover, $ \mu $ is supported on $ K $.  We call $ \mu $ the {\it stationary measure} associated with $ \{S_i\}_{i=1}^\ell $ and $(p_1,\ldots, p_\ell)$.
In particular, $ \mu $ is said to be {\it self-similar} if $ S_{i} $ are all similarities and {\it self-affine} if $ S_{i} $ are all affine maps.

It is known that every self-similar measure $\mu$ on $\R^d$ is exact dimensional  (see \cite{FengHu2009}), that is to say, there exists a constant $C$ such that
$$
\lim_{r\to 0}\frac{\log \mu(B_r(x))}{\log r} =C
$$
for $\mu$-a.e.~$x\in \R^d$, where $B_r(x)$ stands for the closed ball of radius $r$ centred at $x$.  We write $\dim(\mu)$ for this constant and call it the {\it dimension} of $\mu$.

One of the fundamental questions in fractal geometry is to determine the dimension of self-similar  measures.  So far this question has been well-understood when the underlying IFS satisfies the open set condition \cite{Hutchinson1981} or the exponential separation condition \cite{Hochman2014, Hochman2015}. However, in the general overlapping case,  although there are some significant advances in recent years (see e.g. \cite{Rapaport2020a, Rapaport2020b, Varju2019} and the survey papers \cite{Hochman2018, Varju2018})  the question still remains wide open.

The present paper aims to provide some methods to estimate the dimension of self-similar measures from below and above. To state our result, define $ \varphi :[0,\infty)\to\bfR$ by $ \varphi(x)=-x\log x $. Set $\R_+=[0,\infty)$.   Define $f=f_\ell\colon \R_+^\ell\to \R$ by
\begin{equation}
\label{e-f1.1}
f(x_{1},\ldots,x_{\ell})=(x_{1}+\cdots+x_{\ell}) \sum_{i=1}^{\ell}\varphi\left(\dfrac{x_{i}}{x_{1}+\cdots+x_{\ell}}\right).
\end{equation}
The function $f$ is monotone increasing on $\R_+^\ell$ (see Lemma \ref{lmm2.3}).

For a Borel probability measure $\eta$ on $\R^d$, a finite collection $\huaD$ of Borel subsets of $\R^d$ is said to be a {\it finite Borel partition of $\R^d$ with respect to $\eta$} if
$\eta\left(\bigcup_{D\in \huaD}D\right)=1$ and $\eta(D\cap D')=0$ for  different elements $D, D'\in \huaD$. Our first result is the following.

\begin{theorem}
\label{thm-1.0}
Let $\mu$ be the self-similar measure associated with an IFS $\{S_i\}_{i=1}^\ell$ on $\R^d$ and a probability vector $(p_1,\ldots, p_\ell)$. Let $\rho_i$ denote the contraction ratio of $S_i$, $i=1,\ldots, \ell$.  Then the following properties hold.
\begin{itemize}
\item[(i)] For any finite Borel partition $\huaD$ of $\R^d$ with respect to $\mu$,
\begin{equation}
\label{e-p1.1}
\dim (\mu)\geq \frac{\left(\sum_{i=1}^\ell -p_i\log p_i\right) -\sum_{D\in \huaD} f\left(p_1\mu(S_1^{-1}D),\ldots, p_\ell\mu(S_\ell^{-1}D)\right) }{\sum_{i=1}^\ell -p_i\log \rho_i};
\end{equation}
consequently, if $p_i\mu(S_i^{-1}D)\leq y_i(D)$ for $1\leq i\leq \ell$ and $D\in \huaD$, then
\begin{equation}
\label{e-p1.2}
\dim(\mu)\geq \frac{\left(\sum_{i=1}^\ell -p_i\log p_i\right) -\sum_{D\in \huaD} f\left(y_1(D),\ldots, y_\ell(D)\right) }{\sum_{i=1}^\ell -p_i\log \rho_i}.
\end{equation}
\item[(ii)] Let $(\huaD_n)$ be a sequence of finite Borel partitions of $\R^d$ with respect to $\mu$ so that $\max_{D\in \huaD_n}{\rm diam}(D)\to 0$ as $n\to \infty$. Then
\[ \dim(\mu)=\dfrac{\left(\sum_{i=1}^\ell -p_i\log p_i\right) -\lim_{n\to \infty} \sum_{D\in \huaD_n} f\left(p_1\mu(S_1^{-1}D),\ldots, p_\ell\mu_\ell(S_\ell^{-1}D)\right) }{\sum_{i=1}^\ell -p_i\log \rho_i}.
\]
\end{itemize}
\end{theorem}

The above theorem is based on a result of the first author and Hu \cite{FengHu2009} which  states that the dimension of a self-similar measure can be expressed in terms of the projection entropy.  The reader is referred to Section~\ref{S-2} for the involved notation and Theorem \ref{thm-2.1} for the details of this result.
Theorem~\ref{thm-1.0} then follows directly from some lower bound estimates on the projection entropy.  It provides a valid way to estimate  the dimension of general self-similar measures from below. Two examples (see Examples \ref{exam-4.1}--\ref{exam-4.2}) are given to illustrate this method.

As an interesting application,  we can apply the above method to obtain a new uniform lower bound on the dimension of Bernoulli convolutions.
Recall that for each  $\beta\in (1,2)$ the Bernoulli convolution  with parameter $\beta$, say $\mu_\beta$,  is the self-similar measure
associated with the IFS $\{\beta^{-1}x, \beta^{-1}x+1-\beta^{-1}\}$ on $\R$ and the probability vector $(1/2, 1/2)$.
Bernoulli convolutions are one of the most studied examples of overlapping self-similar measures and they are of great interest in fractal geometry (see e.g.~the survey articles \cite{PSS2000, Varju2018}).   It is known that $\dim(\mu_\beta)=1$ if $\beta$ is a transcendental number \cite{Varju2019}
and $\dim(\mu_\beta)<1$ if $\beta$ is a Pisot number \cite{Garsia1963}. Recall that a Pisot number  is an algebraic
integer all of whose Galois conjugates are inside the unit disk. So far Pisot numbers in $(1,2)$ are the only known examples of parameters $\beta$ with $\dim(\mu_\beta)<1$. As for other algebraic numbers, it is known that $\dim(\mu_\beta)=1$ if $\beta$ is an algebraic number which is not a root of a polynomial of coefficients 0 and $\pm 1$ \cite{Hochman2014}, or $\beta$ is an algebraic number with relatively large Mahler measure \cite{BV2020}, or $\beta$ is among some concrete examples of algebraic numbers with small degree \cite{AFKP2020,HKPS2019}.

In \cite{HareSidorov2018}  Hare and Sidorov  showed that $\dim (\mu_\beta)\geq 0.82$ for all $\beta\in (1,2)$. Their algorithm is based on the estimation for the maximal growth rate of overlapping times of the underlying IFS under iterations.  By applying the algorithm in Theorem \ref{thm-1.0}, we can provide a new uniform lower bound on the dimension of Bernoulli convolutions  (see Theorem \ref{thm-1.1}).
Before completing the writing of this paper, we were aware of a very recent independent work \cite{KPV2021} by Kleptsyn,  Pollicott and  Vytnova, who obtained a uniform lower bound  $0.96399$ on the Hausdorff dimension of Bernoulli convolutions through a different approach by estimating the $L^2$-dimension of Bernoulli convolutions from below.

 To state our result, let $\beta_3\approx 1.839286755214161$ be the tribonacci number, i.e., the largest root of the polynomial $x^3-x^2-x-1$.
A computable theoretical formula for $\dim\mu_{\beta_3}$ (expressed in a series) was independently obtained by  the first author \cite{Feng1999, Feng2005} and  Grabner et al.~\cite{Grabner2002}, with a corresponding numerical estimation
    \begin{equation}
    \label{e-1.0'}
    \dim(\mu_{\beta_3})\approx 0.98040931953 \pm 10^{-11}.
    \end{equation}

Now we are ready to state our uniform lower bound\footnote{Our algorithm and a previously obtained  uniform lower bound 0.980368 on $\dim\mu_\beta$,  were announced by the first author in the conference ``Number theory and dynamics'' at Cambridge in March 2019.}.

\begin{theorem}
\label{thm-1.1}
$
\dim(\mu_\beta)\geq 0.98040856
$
for all $\beta\in (1,2)$.
Moreover,  $\dim(\mu_\beta)>\dim(\mu_{\beta_3})$ if  $$\beta\in (\sqrt{2},2)\backslash (\beta_3-10^{-8},  \; \beta_3+10^{-8}).$$
\end{theorem}

It can be proved (see Lemma~\ref{lem-6.3}) that there exists $\beta_*\in (\sqrt{2},2)$ such that $$\dim(\mu_{\beta_*})=\inf_{\beta\in (1,2)}\dim (\mu_\beta).$$
 According to Theorem \ref{thm-1.1},  $\beta_*\in (\beta_3-10^{-8},  \; \beta_3+10^{-8})$.  This leads to the following.

\begin{conjecture}
$\beta_*=\beta_3$. Moreover, $\dim(\mu_\beta)>\dim(\mu_{\beta_3})$ if $\beta\in (1,2)\backslash \{\beta_3\}$.
\end{conjecture}
In the remaining part of this section, we turn to the question how to  estimate $\dim(\mu_\beta)$ with small error when $\beta$ is  a Pisot number. So far, this question  has only been answered in the special case when $\beta=\beta_n$, $n=2,3,\ldots$,  where $ \beta_{n} $ is the largest root of the polynomial $ x^{n}-x^{n-1}-x^{n-2}-\cdots-x-1 $. In such situation, there are computable theoretical formulas for $ \dim(\mu_{\beta_{n}})$; see  \cite{AlexanderZagier1991} for the case when $n=2$ and \cite{Feng1999, Feng2005, Grabner2002} for the general case. This is due to the following special property of $ \mu_{\beta_{n}} $:  they are (locally) self-similar measures associated with  infinite IFSs with no overlaps (see e.g.~\cite{Feng1999, Feng2005}). However,  this property seems not to be generic in the Pisot cases.

In \cite{Lalley1998} Lalley showed that for each Pisot number $\beta\in (1,2)$, $ \dimmubeta $ can be expressed in  terms of the top Lyapunov exponent of a sequence of random matrix products.  Although Lalley provided an algorithm for the construction of these matrices, but the computation of Lyapunov exponents is a very difficult problem and  Lalley only  provided some numerical estimates on the dimension of the standard Bernoulli convolution (and its biased versions) associated with $\beta_2$.

Built on an early work of the first author \cite{Feng2003} and the thermodynamic formalism for matrix products,  for each Pisot number $\beta\in (1,2)$, in what follows  we will express $ \dimmubeta $ in terms of the entropy of the equilibrium measure for certain matrix pressure function, and give some computable upper bounds on $ \dimmubeta $ in terms of conditional entropies.

Recall that in  \cite{Feng2003},  for a given Pisot number $ \beta\in (1,2)$, we can construct a finite family of $d\times d$ non-negative matrices
$ A_{1}, \ldots,A_{k}$, where $d$ and $k$ depend on $\beta$,  so that \begin{equation}
\label{e-H}
H:=\sum_{i=1}^k A_{i} \end{equation}
 is irreducible (i.e., there exists an integer $p$ such that $\sum_{j=1}^p H^j$ is a strictly positive matrix), and moreover,
\begin{equation}
\label{e-dim}
\dim(\mu_\beta)=-\frac{P'(1)}{\log \beta},\qquad \tau_{\mu_\beta}(q)=-\frac{P(q)}{\log \beta}\;\mbox { for }\;q>0,
\end{equation}
where $P\colon (0,\infty)\to \R $ denotes the  pressure function associated with  $(A_{1},\ldots,A_{k}) $, which is defined by
\begin{equation}
\label{e-P}
P(q) = \nliminf \dfrac{1}{n} \log\left(\sum_{i_{1}\cdots i_{n}\in \{1,\ldots, k\}^n}\normed{A_{i_{1}}\cdots A_{i_{n}}}^{q} \right),
\end{equation}
here $\|\cdot\|$ is the standard matrix norm,  and $\tau_{\mu_\beta}(q)$ stands for the $L^q$-spectrum of $\mu_\beta$, which is defined by
\[ \tau_{\mu_\beta}(q)=\liminf_{r\to 0}\dfrac{\log\sum_{Q\in \mathcal Q_r}{\mu_\beta(Q)}^q}{\log r},\qquad q>0, \]
where $\mathcal Q_r:=\{ [jr, (j+1)r):\; j\in \mathbb{Z} \}$. See \cite[Theorems 1.3 and 5.2]{Feng2003}. By definition $\tau_{\mu_\beta}(1)=0$, it follows from \eqref{e-dim} that  $P(1)=0$ which implies that $1$ is the largest eigenvalue of $H$. Since $H$ is an irreducible non-negative matrix, it has  left and right positive eigenvectors associated to the eigenvalue $1$ (see e.g.~\cite[Theorem~8.4.4]{HornJohnson1985}).   Let ${\bf v}_L$ , ${\bf v}_R$ be the left and right positive eigenvectors of $H$ such that ${\bf v}_L\cdot{\bf v}_R=1$. Define a Borel probability measure $\eta$  on $\Sigma=\{1,\ldots, k\}^\N$ by
\begin{equation}
\label{e-G}
\eta\left([i_{1}\cdots i_{n}]\right) = {\bf v}_L  A_{i_{1}}\cdots A_{i_{n}} {\bf v}_R \quad \mbox{ for }\;  n\in \N\; \mbox{ and }\; i_1\cdots i_n\in \{1,\ldots, k\}^n,
\end{equation}
where $[i_{1}\cdots i_{n}]:=\left\{x=(x_j)_{j=1}^\infty\in \Sigma:\; x_j=i_j \mbox{ for } 1\leq j\leq n\right\}.$

It is readily checked that $\eta$ is indeed a probability measure and moreover, $\eta$ is $\sigma$-invariant, where $\sigma\colon \Sigma\to \Sigma$ is the left shift map defined by $(x_j)_{j=1}^\infty\mapsto (x_{j+1})_{j=1}^\infty$.
Let $h_\eta(\sigma)$ denote the measure-theoretic entropy of $\eta$ with respect to $\sigma$ (see \cite{Walters1982} for a definition). Now we can state our last result.

\begin{theorem}
\label{thm-1.4}
Let $\beta\in (1,2)$ be a Pisot number, and let $\eta$ be constructed as above. Then
$
\dim(\mu_\beta)=\dfrac{h_\eta(\sigma)}{\log \beta}.
$
Consequently, for each $n\in \N$,
\begin{equation}
\label{e-1.7}
\dim(\mu_\beta)\leq \frac{\sum_{I\in \{1,\ldots,k\}^{n+1}}\varphi\left(\eta([I]) \right) -
	\sum_{J\in \{1,\ldots,k\}^{n}}\varphi\left(\eta([J]) \right)}{\log \beta},
\end{equation}
where $\varphi(x)=-x\log x$.
\end{theorem}

The inequality \eqref{e-1.7} provides a sequence of upper bounds on $\dim(\mu_\beta)$ for a general Pisot number $\beta$. When $\beta$ is of small degree, the dimension and the number of the constructed matrices $A_i$ are not very large, so one can use \eqref{e-1.7} to obtain reasonable upper bounds on $\dim(\mu_\beta)$. In the meantime one can use the algorithm in Theorem \ref{thm-1.0} to estimate $\dim(\mu_\beta)$ from below. Hence in this situation one can provide the estimates on $\dim(\mu_\beta)$ with small error. In Section \ref{S-7}, we list our computational results on $\dim(\mu_\beta)$ for some Pisot numbers of degree $3$ or $4$. For instance, let $\beta\approx
  1.465571231876768$ be the largest root of the polynomial $x^3-x^2-1$,
our computation shows that $\dim(\mu_\beta)\approx 0.99954470$ with an error $\leq 10^{-8}$.

It is worth pointing out that  the method of estimating projection entropies can also be used to find lower bounds on the dimension of  self-affine measures generated by diagonal affine iterated function systems. In Section \ref{S-8}, we will provide a more detailed justification of this fact and give an example.

The paper is organised as follows:  In Section \ref{S-2} we give the definition of projection entropy and present a result in \cite{FengHu2009}  on the dimension of self-similar measures. In Section \ref{S-3}, we give some upper bound estimates on conditional entropies and prove Theorem \ref{thm-1.0}.
In Section \ref{S-4}, we give an example to illustrate the application of Theorem \ref{thm-1.0}.  In Section 5,  we provide an algorithm to produce uniform lower bounds on $\dim(\mu_\beta)$ over small intervals of $\beta$,   and give a computer-assisted proof of Theorem \ref{thm-1.1}.  In Section \ref{S-6}, we investigate the asymptotic properties of
$\dim(\mu_\beta)$ when $\beta$ is close to   $2$.  In Section \ref{S-7}, we prove Theorem \ref{thm-1.4} and give computational results on $\dim(\mu_\beta)$ for some Pisot numbers of degree $3$ or $4$. In Section \ref{S-8}, we give some final remarks.

\section{Preliminary}
\label{S-2}
In this section, we introduce the concept of projection entropy for a Bernoulli product measure associated with an IFS and present  a result in \cite{FengHu2009} on the dimension of self-similar measures.

Let $\{S_i\}_{i=1}^\ell$ be an IFS on $\R^d$ with attractor $K$, and let $ (\Sigma,\sigma ) $ be the one-sided full shift over the alphabet $ \{1,\ldots, \ell\}$.  That is, $ \Sigma=\{1,\ldots, \ell\}^\N$ and $\sigma $ is the left shift on $ \Sigma $ defined by
\[ \sigma ((x_{n})_{n=1}^{\infty}) = (x_{n+1})_{n=1}^{\infty}. \]
Let $ \pi: \Sigma\to\eucild{d} $ be the canonical coding map associated with the IFS  $\{S_i\}_{i=1}^\ell$. That is,
  \[ \pi(x) = \lim_{n\to\infty} S_{x_{1}}\circ \cdots \circ S_{x_{n}}(0),\qquad x=(x_n)_{n=1}^\infty.\]
Let $ m=\prod_{n=1}^{\infty}\{p_{1},\ldots, p_{\ell}\} $ be the Bernoulli product measure on $ \Sigma $ and $\mu=m\circ \pi^{-1}$,  that is,
\[ \mu (A) = m(\pi^{-1}(A)) \]
for any Borel subset $ A $ of $ \eucild{d} $. Clearly, $\mu$ is supported on $K$. It is well-known  (\cite{Hutchinson1981}) that $\mu$ is the unique Borel probability measure on $\R^d$ such that
\begin{equation}
\label{e-S2.1}
\mu=\sum_{i=1}^\ell p_i\mu\circ S_i^{-1}.
\end{equation}

Let $ \calP =\lbrace[i]\colon i=1,\ldots,\ell\rbrace $ be the natural partition of $ \Sigma $, where
 \[ [i]:= \lbrace x = (x_{n})_{n=1}^{\infty}\in
\Sigma:\; x_{1}=i\rbrace. \]
The following definition was introduced in \cite{FengHu2009}.

\begin{definition}\label{def1.1}
	The projection entropy of $ m $ under $ \pi $ is
	\[ h_{\pi}(\sigma,m) := H_{m}(\calP) - H_{m}(\calP\vert \pi^{-1}\calB(\eucild{d}) ), \]
	where $ H_{m}(\calP) = \sum_{i=1}^{\ell}-p_{i}\log p_{i} $ is the entropy of the partition $ \calP $, $\calB (\eucild{d})$ is the Borel $\sigma$-algebra on $\R^d$, and $ H_{m}(\calP\vert\pi^{-1}\calB(\eucild{d})) $ is the conditional entropy of $ \calP $ given $ \pi^{-1}\calB (\eucild{d}) $.
\end{definition}

The reader is referred to \cite{Walters1982} for the definition of conditional entropy. The concept of projection entropy plays a crucial role in the dimension theory of IFS \cite{FengHu2009}. In particular,  it can be used to characterize the dimension of self-similar measures.

 \begin{theorem}[{\cite[Theorem~2.8]{FengHu2009}}]
 \label{thm-2.1}
	Let $ \{S_i\}_{i=1}^\ell$  be an IFS consisting of similarities.  Let $ \rho_{i}$ be the contraction ratio of $S_i$, $i=1,\ldots, \ell$. Then $ \mu $ is exact dimensional and
	\[ \dim (\mu)= \dfrac{h_{\pi}(\sigma,m)}{\lambda}=\dfrac{H_{m}(\calP) - H_{m}(\calP\vert\pi^{-1}\calB(\eucild{d}))}{\lambda}, \]
	where $ \lambda:=-\sum_{i=1}^{\ell}p_{i}\log\rho_{i} $.
\end{theorem}

Since $ H_{m}(\calP) $ and $ \lambda $ are easy to compute, in order to estimate $ \dim(\mu)$, it is sufficient to estimate $ H_{m}(\calP\vert\pi^{-1}\calB(\eucild{d}))  $.

In the next section, we provide an algorithm to estimate $ H_{m}(\calP\vert\pi^{-1}\calB(\eucild{d})) $ from above for any given IFS on $\R^d$.

\section{Upper bounds  on $H_{m}(\calP\vert\pi^{-1}\calB(\eucild{d})) $ and the proof of Theorem \ref{thm-1.0}}
\label{S-3}

Let $ \{S_i\}_{i=1}^\ell$ be an IFS on $ \eucild{d} $ with attractor $ K $. Here we only assume that $ S_{i} $ are contracting. Let $\pi$, $m$, $\mu$ and $\mathcal P$ be defined as in Section \ref{S-2}.
Recall that for a finite Borel partition $\mathcal A$ of $\Sigma$,
\begin{equation}
\label{e-H3.1}
H_m(\mathcal P \vert \mathcal A):=H_m(\mathcal P\vee \mathcal A)-H_m(\mathcal  A),
\end{equation}
where $\mathcal P\vee \mathcal A :=\{P\cap A\colon P\in \mathcal P,\; A\in \mathcal A\}$; see e.g. \cite[\S4.3]{Walters1982}.

In this section we will provide upper bounds on the condition entropy $H_{m}(\calP\vert\pi^{-1}\calB(\eucild{d}))$, and prove Theorem \ref{thm-1.0}.  At the end of this section, we also give an iterative algorithm to estimate $\mu(A)$ from above with small error for a given bound Borel set $A\subset \R^d$.

The following result is our starting point.

\begin{lem}\label{prop2.1}
	\begin{itemize}
		\item[(i)]  Let $ \huaD $ be a finite Borel partition of $\eucild{d}$. Then
		\[ \conditionE\leq H_{m}(\calP\vert\pi^{-1} \huaD). \]

		\item[(ii)] Let $ \lbrace\huaD_{n}\rbrace $ be a sequence of finite Borel partitions of $K$ with
		\[\diam(\huaD_{n}):= \sup_{D\in\huaD_{n}}\diam(D) \to 0.  \]
Then
$\conditionE= \nliminf \conditionD{n}$.
	\end{itemize}
\end{lem}

\begin{proof}
	Let $ \huaD $ be a finite Borel partition of $ \eucild{d} $. Then the \sigalgebra  generated by $ \pi^{-1}\huaD $ is a \subsigalgebra of $ \pi^{-1} \calB(\eucild{d}) $. By \cite[Theorem 4.3(v)]{Walters1982},
	\[ \conditionE\leq\conditionD{}. \]
	 This proves (i).
	
	To prove (ii), let $ \lbrace\huaD_{n}\rbrace $ be a sequence of finite Borel partitions of $ K $ with ${\rm diam}(\huaD_{n})\to 0 $ as $ n\to\infty $. Let $ \varepsilon>0 $. In view of (i), it suffices to show that
	\begin{equation}
\label{e-2.1}
 \nlimsup \: \conditionD{n}\leq\conditionE +\varepsilon.
\end{equation}

To this end, take a sequence $ \lbrace\huaA_{n}\rbrace$ of finite Borel partitions of $K$ such that
	\[ \huaA_{n+1}\geq\huaA_{n},\quad {\rm diam}(\huaA_{n})\to 0, \quad \mbox{and}\quad \huaA_{n}\uparrow \calB(K),\]
	where $ \huaA\geq\huaC $ means that any element in $ \huaC $ is the union of elements in $ \huaA $, and $\calB(K)$ stands for the  \sigalgebra of Borel subsets of $K$.
		Since the range of $\pi$ is $K$,  it follows that $$\sigma(\pi^{-1}\huaA_{n})\uparrow\pi^{-1}\calB(K)=\pi^{-1}\calB(\eucild{d}).$$
Hence by  \cite[Theorem 4.7]{Walters1982},
	$H_{m}(\calP\vert\pi^{-1}\huaA_{n})\downarrow \conditionE.$
	Take a large positive  integer $k$ such that
	\[ \conditionA{k}\leq\conditionE+\varepsilon.\]
	Write $ \huaA_k= \lbrace A_{1},\ldots,A_{M}\rbrace $.
	
	Since ${\rm diam}(\huaD_{n})\to 0 $ as $ n\to\infty $, by \cite[Lemma 1.23]{Bowen1975} there are partitions $ \huaE_{n}=\lbrace E_{1}^{n},  \ldots, E_{M}^{n}\rbrace $ of $K$ such that
	\begin{itemize}
				\item[(a)] Each $ E_{i}^{n} $ is a union of members of $ \huaD_{n} $;
		\item[(b)] $ \nliminf\mu(E_{i}^{n}\bigtriangleup A_{i}) =0  $ for each $1\leq  i\leq M $, where $A\bigtriangleup B:=(A\backslash B)\cup (B\backslash A)$. 	\end{itemize}

	Notice that for any $ 1\leq j\leq \ell $ and $ 1\leq i \leq M $,
	\begin{equation}\nonumber
	\begin{split}   m(([j]\cap\pi^{-1}(E_{i}^{n}))\bigtriangleup([j]\cap\pi^{-1}(A_{i})))  &\leq m(\pi^{-1}(E_{i}^{n})\bigtriangleup\pi^{-1}(A_{i}))
	\\ &=\mu(E_{i}^{n}\bigtriangleup A_{i})\to 0,\:\quad  \text{as} \: n\to\infty.
	\end{split}
	\end{equation}
    It follows that
	\[ \lim_{n\to \infty}\conditionhuaE{n}=\conditionA{k}\leq \conditionE + \varepsilon.\]
	Since $ \huaD_{n}\geq\huaE_{n} $, by \cite[Theorem 4.3(v)]{Walters1982} we have $\conditionD{n}\leq\conditionhuaE{n}$, therefore
		\[ \nlimsup \: \conditionD{n}\leq \lim_{n\to \infty}\conditionhuaE{n}\leq \conditionE +\varepsilon. \]
This proves \eqref{e-2.1} .
\end{proof}
\begin{remark}
\label{rem-3.2}
Let $ \huaD $ be a finite Borel partition of  $\R^d$ with respect to $\mu$ (cf. Section~\ref{S-1}). Since $\mu$ is supported on $K$, there exists a finite Borel partition $\huaD'$ of $K$ such that
for each $D\in \huaD$ there exists $D'\in \huaD'$ so that $\mu(D\bigtriangleup D')=0$. This implies that  $H_m(\mathcal P|\pi^{-1}\huaD)= H_m(\mathcal P|\pi^{-1}\huaD')$. It follows from  Lemma \ref{prop2.1}(i) that  $ \conditionD{} $ is an upper bound of $ \conditionE $.
\end{remark}
Below we discuss  how to estimate $ \conditionD{} $ for a given finite Borel partition $\huaD$ of $\R^d$ with respect to $\mu$.

\begin{lem}\label{lmm2.2}
	For  $D\in {\mathcal B}(\eucild{d})$ and $i\in\lbrace1,\ldots,\ell \rbrace $,
	\[ m([i]\cap\pi^{-1}(D)) =p_{i} \mu(S_{i}^{-1}(D)). \]
\end{lem}

 \begin{proof}
    Observe that for $x=(x_{n})_{n=1}^{\infty}\in \Sigma$,
	\begin{align*}
		x\in [i]\cap\pi^{-1}(D)
	& \iff x_{1} = i,\; \pi x\in D\\
	& \iff x_{1} = i,\;  S_{i}(\pi\sigma x)\in D\\
	& \iff x_{1} = i,\; x \in \sigma^{-1}\pi^{-1}(S_{i}^{-1}D)\\
	& \iff x \in [i] \cap \sigma^{-1}\pi^{-1}(S_{i}^{-1}D).
	\end{align*}
Hence $ [i]\cap\pi^{-1}(D)=[i]\cap\sigma^{-1}\pi^{-1}(S_{i}^{-1}D) $.
It follows that
	\begin{align*}
		m([i]\cap\pi^{-1}(D))&=m([i]\cap\sigma^{-1}\pi^{-1}(S_{i}^{-1}D))\\
	&= p_{i}m(\pi^{-1}(S_{i}^{-1}D))\\
	&=p_{i}\mu(S_{i}^{-1}D),
		\end{align*}
	where in the second equality we used the property that $ m([i]\cap\sigma^{-1}A)=p_{i}m(A) $ for any Borel subset  $ A$ of $\Sigma $.
\end{proof}

Let $f\colon\R_+^\ell\to \R$ be defined as in \eqref{e-f1.1}.

\begin{lem}\label{lmm2.3} For every $1\leq i\leq \ell$,
	$\frac{\partial f}{\partial x_{i}}\geq 0 $.
	Consequently, $f$ is monotone increasing over $\R_+^\ell$ in the sense that
	\[ f(x_{1}+\varepsilon_{1},\ldots,x_{\ell}+\varepsilon_{\ell})\geq  f(x_{1},\ldots,x_{\ell}) \]
	for any $ x_{1},\ldots,x_{\ell},\varepsilon_{1},\ldots,\varepsilon_{\ell} \geq 0$.
\end{lem}

\begin{proof}
   A direct computation shows that  for each $i$,
	\[  \frac{\partial f}{\partial x_{i}} (x_1,\ldots, x_\ell)= \log(x_{1}+\cdots+x_{\ell})-\log x_{i} \geq 0,\]
from which we obtain the desired inequality for $f$.
\end{proof}

\begin{lem}\label{lem-3.5} Let $\huaD$ be a finite partition of $\Bbb R^d$ with respect to $\mu$. Then
	\[ \conditionD{}=\sum_{D\in\huaD}f(\mu_{1}(D),\ldots,\mu_{\ell}(D)), \]
	where $ \mu_{i}(D):=p_{i}\mu(S_{i}^{-1}D) $.
\end{lem}

\begin{proof} By \eqref{e-H3.1},
	\begin{equation}\nonumber
	\begin{split}
	\conditionD{}&=H_{m}(\calP\vee \pi^{-1}\huaD)-H_{m}(\pi^{-1}\huaD)\\
	&= \sum_{D\in\huaD}\left( \left(\sum_{i=1}^{\ell}\varphi(m([i]\cap\pi^{-1}D))\right)-\varphi(m(\pi^{-1}D))\right)
	\\
	&= \sum_{D\in\huaD}\left(\left(\sum_{i=1}^{\ell}\varphi(p_{i}\mu(S_{i}^{-1}D))\right)-\varphi(\mu(D))\right)\\
	& =\sum_{D\in\huaD}f(\mu_{1}(D),\ldots,\mu_{\ell}(D)),
	\end{split}
	\end{equation}
	where we used Lemma \ref{lmm2.2} in the third equality, and the fact that $\mu(D)=\sum_{i=1}^\ell\mu_i(D)$ (which follows from \eqref{e-S2.1}) in the fourth equality.
\end{proof}

\begin{corollary}\label{corol2.5} Let $\huaD$ be a finite partition of $\Bbb R^d$ with respect to $\mu$.
	Suppose that $$ p_{i}\mu(S_{i}^{-1}D)\leq y_{i}(D)\quad \mbox{ for every } i\in\lbrace1,\ldots,\ell\rbrace \mbox{ and } D\in\huaD. $$
	Then
	$\conditionD{}\leq\sum_{D\in\huaD}f(y_{1}(D),\ldots,y_{\ell}(D))$.
\end{corollary}

\begin{proof}
	This follows immediately by combining Remark \ref{rem-3.2}, Lemmas \ref{lmm2.3} and \ref{lem-3.5}.
\end{proof}

Now we are ready to prove Theorem \ref{thm-1.0}.

\begin{proof}[Theorem \ref{thm-1.0}]
Part (i) follows by combining Theorem \ref{thm-2.1}, Remark \ref{rem-3.2} and Lemma 3.5. Part (ii) follows by combining Theorem \ref{thm-2.1}, Lemma
\ref{prop2.1},  Remark \ref{rem-3.2} and Lemma 3.5.
\end{proof}

Lemma \ref{prop2.1} (also Remark \ref{rem-3.2}) and Corollary \ref{corol2.5} provide us the following theoretical way to estimate $ \conditionE $ from above: first choose a finite Borel partition $ \huaD $ of $\R^d$ with respect to $\mu$, and find $ y_{i}(D)\geq p_i\mu(S_i^{-1}(D))$ for each $i\in \{1,\ldots, \ell\}$ and $D\in \huaD$.
Then
\begin{equation}
\label{e-U}
\conditionE \leq  \sum_{D\in\huaD}f(y_{1}(D),\ldots,y_{\ell}(D)).
\end{equation}
In the remaining part of this section, we discuss how to find $y_{i}(D)\geq p_i\mu(S_i^{-1}(D))$ with $ y_{i}(D)- p_i\mu(S_i^{-1}(D))$ being small. This reduces to the following.

\begin{problem}
How to estimate $ \mu(A) $ from above with small error for a given bounded Borel set $ A\subset\eucild{d}$?
\end{problem}

In what follows,  we give an answer to the above problem  in the special case when $p_1=\ldots=p_\ell=1/\ell$, by providing a simple iterative algorithm.

\begin{algorithm}
\label{alg-3.8}
	Let $A\subset \R^d$ be a given bounded Borel set. Take $ \gamma >0  $ and a positive integer $ L $.
	Choose a closed ball $ B\subset \R^d$ such that $ S_{i}(B)\subset B$, $i=1,\ldots,\ell$. Then $ K\subset B$. Below we inductively construct a finite sequence $(u_n)_{n=0}^{L^*}$ of non-negative rational numbers,  and a finite sequence $(\Lambda_n)_{n=0}^{L^*}$ of finite ``subsets'' of $\R_+\times \mathcal B(\R^d)$  for some $L^*\leq L$, and a non-negative rational number $u^*$ (which is our upper bound for $\mu(A)$.  Here for convenience,  we allow that the elements in $\Lambda_n$ may be the same, but they are counted as different elements.
	\begin{itemize}
		\item[(1)] Set $ u_{0}=0 $ and $ \Lambda_{0}=\left \lbrace (1, A)\right \rbrace$.
		
\item[(2)] Suppose we have obtained $u_n$ and  $\Lambda_{n} $ for some $n<L$.  If $\Lambda_n=\emptyset$ then set $u^*=u_n$, $L^*=n$ and  complete the algorithm. Otherwise, we set $u_{n+1}'=u_n$ and
\begin{equation}
\label{e-3.4e} \Lambda_{n+1}^*=\bigcup_{(t,E)\in \Lambda_n} \left\{ \left( t/\ell,\overline{V}_\gamma(S_{i}^{-1}E)\cap B\right):\; i=1,\ldots, \ell \right\},
\end{equation}
where  $ \overline{V}_{\gamma}(E) := \lbrace x\colon d(x,E)\leq \gamma \rbrace $ for $E\neq\emptyset$, and $\overline{V}_\gamma(\emptyset)=\emptyset$. Write $$\Lambda_{n+1}^*=\{(t_j, E_j):\; j=1,\ldots, k_{n+1}\}.$$ Keep in mind that we don't require  $(t_j, E_j)$ to be distinct for different $j$. For $j=1,\ldots, k_{n+1}$, we apply the following operations consecutively:
\begin{itemize}
\item[(a)] If $E_j=\emptyset$, then remove the element $(t_j, E_j)$ from $\Lambda_{n+1}^*$.

\item[(b)] If  $E_j= B$, then remove the element $(t_j, E_j)$ from $\Lambda_{n+1}^*$, and add $t_j$ to $u_{n+1}'$.
\end{itemize}
Finally set  $u_{n+1}=u_{n+1}'$ and $\Lambda_{n+1}=\Lambda_{n+1}^*$.

\item[(3)]  Repeating the above procedures until we obtain $u_L$ and $\Lambda_{L} $.
	 Finally we let  $L^*=L$ and
		\begin{equation} \label{e-3.5e}
u^{\ast}=u_{L}+\sum_{(t,E)\in\Lambda_{L}}t, \end{equation}
		    	then we complete the algorithm.
 \end{itemize}
 The rational number $u^{\ast} $ obtained above is our upper bound for $ \mu(A) $.
\end{algorithm}	

\begin{remark}\label{rmk2.6}
\begin{itemize}
\item[(i)]
The number $L$ in the above algorithm is called the iteration time.
\item[(ii)]  For each element $(t,E)$ of $\Lambda_n$, $t$ is of the form $p2^{-n}$ and can be saved as $(n,p)$, where $p$ is a positive integer $\leq 2^n$.  The addition operations involved in the calculations of $u_n$ and $u^*$ can be conducted by performing certain integer arithmetic operations, and bring no rounding errors.
\item[(iii)] The introduction of $\gamma$ in the above algorithm is used to compensate the rounding errors in the computation of $S_i^{-1}(E)$, and it should be selected according to the computation precision.
\end{itemize}
\end{remark}

In the general case when $(p_1,\ldots, p_\ell)$ is an arbitrary probability vector, to estimate $\mu(A)$ we need to modify Algorithm 3.8 accordingly. The main change is to replace the definition of $\Lambda_{n+1}^*$ in \eqref{e-3.4e} by
$$
\Lambda_{n+1}^*=\bigcup_{(t,E)\in \Lambda_n} \left\{ \left(tp_i(1+\gamma),\overline{V}_\gamma(S_{i}^{-1}E)\cap B\right):\; i=1,\ldots, \ell \right\},
$$
In this case, the calculations on both components of $(t,E)$ in $\Lambda_n$  are performed by  floating-point arithmetics. After obtaining $u^*$ from \eqref{e-3.5e}, one still needs to add a suitable error term $\Delta u^*$  to $u^*$ so as to compensate the rounding errors brought by the addition operations involved in the calculations of $u_1,\ldots, u_{L^*}$ and $u^*$.

\section{two examples}\label{S-4}

In this section, we use two examples to illustrate how to apply Theorem \ref{thm-1.0} to get lower bound estimates on the dimension of self-similar measures.
\begin{example}
\label{exam-4.1}
{\rm
Let $ \mu $ be the self-similar measure associated with an IFS $ \lbrace S_{i}\rbrace_{i=1}^{3} $ on $\R$  and the probability vector $ \lbrace\frac{1}{3},\frac{1}{3},\frac{1}{3}\rbrace $, where
\[ S_{1}(x) =\erfenzi x, \; S_{2}(x)=\frac{x+1}{3}, \;S_{3}(x)=\frac{x+3}{4}.\]
Let $K$ be the attractor of $ \lbrace S_{i}\rbrace_{i=1}^{3} $. It is easily checked that the convex hull of $K$ is the interval $[0,1]$.
Let $ \mu $ be the self-similar measure associated with $ \lbrace S_{i}\rbrace_{i=1}^{3} $ and the probability vector $ \lbrace\frac{1}{3},\frac{1}{3},\frac{1}{3}\rbrace $.
Let $ m=\prod_{n=1}^{\infty}\lbrace\frac{1}{3},\frac{1}{3},\frac{1}{3}\rbrace $ be the Bernoulli product measure on $ \Sigma=\lbrace 1,2,3\rbrace^{\bfN} $.
Let
\[ \huaD_1=\left \lbrace \left[0,\frac{1}{3}\right], \left[\frac{1}{3},\Erfenzi\right],\left[\Erfenzi,\frac{2}{3}\right],\left[\frac{2}{3},\frac{3}{4}\right],\left[\frac{3}{4},1\right]\right \rbrace,\] which is the partition of $ \unitinterval $ (with respect to $\mu$) generated by the endpoints of the intervals $ S_{i}\unitinterval, i=1,2,3 $.   This IFS is somehow special, in which the contraction ratio of each map is the reciprocal of an integer, and the translation part is a rational number.     Due to this property, we can use the self-similarity of $\mu$ to compute the precise value of $\mu(S_i^{-1}D)$ for every $i$ and $D\in \huaD_1$.   To see this, applying the self-similarity relation
$\mu=\sum_{i=1}^3 3^{-1}\mu\circ S_i^{-1}$ to the interval $[0,1/3]$, we get
\begin{align*}
\mu([0,1/3])&=3^{-1}\mu([0, 2/3])+3^{-1}\mu([-1,0])+3^{-1}\mu([-3, -5/3])\\
&=3^{-1}\mu([0, 2/3]),
\end{align*}
where in the second equality we used the facts that  $\mu$ is supported on $[0,1]$ and has no atoms. Similarly, we have
\begin{align*}
\mu([0,2/3])&=3^{-1}\mu([0, 4/3])+3^{-1}\mu([-1,1])+3^{-1}\mu([-3, -1/3])\\
&=2/3.
\end{align*}
From the above equalities, we see that $\mu(S_1^{-1}[0,1/3])=2/3$ and $\mu(S_2^{-1}[0,1/3])=\mu(S_3^{-1}[0,1/3])=0$.  Similarly by direct computations with a bare hand  we can obtain the precise value of $\mu(S_i^{-1}D)$ for each $i\in\{1,2,3\}$ and $D\in \huaD_1$; see Table \ref{tab:mu}.

\begin{table}
\begin{center}
\begin{tabular}{c|c|c|c}
$D\in \huaD_1$ & $\mu(S_1^{-1}D)$  & $\mu(S_2^{-1}D)$  & $\mu(S_3^{-1}D)$ \\
\hline
$[0,1/3]$ & 2/3 & 0 & 0 \\
$[1/3,1/2]$  & 1/3   & 1/2  & 0  \\
$[1/2, 2/3]$ & 0   & 1/2  & 0 \\
$[2/3, 3/4]$ &0& 0 & 0  \\
$[3/4, 1]$ & 0 & 0 &  1\\ \hline
\end{tabular}
\\ [2ex]
\caption{Precise values of $\mu(S_i^{-1}D)$ for $i=1,2,3$ and $D\in \huaD_1$.}
\label{tab:mu}
\end{center}
\end{table}
It follows that
\begin{equation}
\label{e-f3}
\sum_{D\in \huaD_1} f\left(3^{-1}\mu(S_1^{-1}D),  3^{-1}\mu(S_2^{-1}D), 3^{-1}\mu(S_3^{-1}D)\right)=\frac{5}{18} \left(\varphi(\frac{2}{5})+\varphi(\frac{3}{5})\right).
\end{equation}
By Theorem \ref{thm-1.0}(i),
\begin{equation}\nonumber
\begin{split}
\dim (\mu) &\geq
\dfrac{\log 3 - \frac{5}{18}\left(\varphi(\frac{2}{5})+\varphi(\frac{3}{5})\right)}{\frac{1}{3}\log 2 +\frac{1}{3}\log 3 +\frac{1}{3}\log 4 } \\
& \approx 0.86058762883316.
\end{split}
\end{equation}

Alternatively, instead of computing the precise values of $\mu(S_i^{-1}D)$ we may use Algorithm \ref{alg-3.8} to estimate $\mu(S_i^{-1}D)$ from above, then use \eqref{e-p1.2} in  Theorem \ref{thm-1.0}(i) to get a lower bound of $\dim (\mu)$.  The computation will be done by using the float-point number type double (binary64) on a standard x64 machine. Below we provide more details.

For a given closed subinterval $A$ of $[0,1]$, to estimate $\mu(A)$ from above,  we apply  Algorithm \ref{alg-3.8} in which we take $B=[0,1]$, $L=40$ and $\gamma=12.2*u$, where $u=2^{-53}\approx 1.1102\times10^{-16}$ denotes the unit roundoff.
A key ingredient of Algorithm \ref{alg-3.8} is the computation of $\Lambda_{n+1}^*$, $n=0,\ldots, L-1$. By \eqref{e-3.4e},
\begin{equation}
\label{e-3.4e*} \Lambda_{n+1}^*=\bigcup_{(t,[a,b])\in \Lambda_n} \left\{ \left(t/3,[S_{i}^{-1}(a)-\gamma, S_{i}^{-1}(b)+\gamma]\cap [0,1]\right):\; i=1,2, 3 \right\},
\end{equation}
where $S_1^{-1}(x)=2x$, $S_2^{-1}(x)=3x-1$ and $S_3^{-1}(x)=4x-3$.

The choice of the specific value $12.2*u$ for $\gamma$ is to guarantee that for every $a\in [0,1]$ and $i\in \{1,2,3\}$,
\begin{equation}\label{e-4.2e}
fl\left(S_i^{-1}(a)-\gamma\right)<S_i^{-1}(a),\qquad fl\left(S_i^{-1}(a)+\gamma\right)>S_i^{-1}(a),
\end{equation}
where $fl(x)$ is the floating-point representation of $x$ (cf. \cite{Higham2002}). For the sake of brevity, we only verify the first inequality in \eqref{e-4.2e} for $i=3$, and leave the  remaining cases to the reader.  Noticing that $S_3^{-1}(a)=4a-3$, by the rules of floating-point number arithmetics (cf. \cite[p.~40]{Higham2002}),
$$
fl(4a-3-\gamma)=((4a(1+\delta_1)-3)(1+\delta_2)-\gamma)(1+\delta_3)
$$
for some real numbers $\delta_1, \delta_2,\delta_3$ with  $|\delta_j|<u$. It follows that
\begin{align*}
fl(4a-3-\gamma)-(4a-3)&=4a\delta_1+(4a-3)(\delta_2+\delta_3)-\gamma+O(u^2)\\
&\leq 4a u+2|4a-3|u-\gamma+O(u^2)\\
&\leq 6u-\gamma-O(u^2)<0.
\end{align*}
which proves the first inequality in \eqref{e-4.2e} for $i=3$.

Next let us take one concrete example $A=[0, 1/2]$ to illustrate the main steps of Algorithm \ref{alg-3.8} for the estimation of $\mu(A)$. We start from $u_0=0$ and $\Lambda_0=\{1, [0,1/2]\}$.   Applying \eqref{e-3.4e*} with $n=0$ yields
$$
\Lambda_1^*=\{(1/3, [0,1]), (1/3, [0, 0.500000000000001]), (1/3, \emptyset)\}.
$$
Removing the elements $(1/3, [0,1])$ and $(1/3, \emptyset)$ from $\Lambda_1^*$ we obtain that $$\Lambda_1=\{(1/3, [0, 0.500000000000001])\}\quad\mbox{and}\quad  u_1=u_0+1/3=1/3.$$ Applying \eqref{e-3.4e*} with $n=1$, we have
$$
\Lambda_2^*=\{(1/9, [0,1]), (1/9, [0, 0.500000000000005]), (1/9, \emptyset)\}.
$$
Removing the element $(1/9, [0,1])$ and $(1/9, \emptyset)$ from $\Lambda_2^*$ we obtain that $$\Lambda_2=\{(1/9, [0, 0.500000000000005])\}\quad\mbox{and}\quad  u_2=u_1+1/9=4/9.$$ Continuing the above procedures, we  obtain $\Lambda_n$ and $u_n$ for $3\leq n\leq 33$, where
$$
u_{33}=\frac{2779530283277771}{3^{33}}, \mbox{ and }\Lambda_{33}=\{(3^{-33},[0,0.7497505160822597])\}.
$$
Then applying \eqref{e-3.4e*} with $n=33$,
$$
\Lambda_{34}^*=\{(3^{-34}, [0,1]), (3^{-34}, [0, 1]), (3^{-34}, \emptyset)\}.
$$
So $\Lambda_{34}=\emptyset$, $L^*=34$,
$$u^*=u_{34}=u_{33}+ 2\times 3^{-34}= \frac{8338590849833315}{16677181699666569}\approx 0.5000000000000018,$$
which is the end of our algorithm.
The above value $u^*$  is the obtained upper bound for $\mu([0,1/2])$.

In Table 2, we list our estimations of $\mu(S_i^{-1}D)$ from above for $D\in \huaD_1$, $i=1,2,3$. Then by \eqref{e-p1.2},
$$
\dim(\mu)\geq \frac{\log 3 -\sum_{D\in \huaD_1} f\left(y_1(D),y_2(D), y_3(D)\right) }{\frac{1}{3}\log 2 +\frac{1}{3}\log 3 +\frac{1}{3}\log 4}\approx 0.8605876288174681.
$$

\begin{table}
\begin{center}
\tiny
\begin{tabular}{c|c|c|c}
$D\in \huaD_1$ & $\mu(S_1^{-1}D)$  & $\mu(S_2^{-1}D)$  & $\mu(S_3^{-1}D)$ \\
\hline
$[0,1/3]$ & 0.6666666666666666 & $8.225263339969959\times 10^{-20}$ & 0.0\\
$[1/3,1/2]$  & 0.33333333333488513 & 0.5000000000000018 & 0.0  \\
$[1/2, 2/3]$ & $1.770354437904526\times 10^{-12}$  & 0.5000000000062046 & 0.0 \\
$[2/3, 3/4]$ &0.0& $1.770354437904526\times 10^{-12}$ &$8.225263339969959\times 10^{-20}$ \\
$[3/4, 1]$ & 0.0 & 0.0 &  1.0\\ \hline
\end{tabular}
\\ [2ex]
\caption{Estimations of $\mu(S_i^{-1}D)$ from above for $i=1,2,3$ and $D\in \huaD_1$.}
\label{tab:mu-above}
\end{center}
\end{table}

In the above computations, if we replace $\huaD_1$ by  $\huaD_n$ ($n\geq 2$), where $\huaD_n$ is the partition of $[0,1]$ generated by the endpoints of the intervals $S_{i_1\ldots i_n}([0,1])$, $i_1\ldots i_n\in \{1,2,3\}^n$, then we can obtain larger lower bounds on $\dim(\mu)$.  In practice,  for $n=2,3,4,5, 6$, we manage to run a program to compute the precise value of $\mu(S_i^{-1}D)$ for each $1\leq i\leq 3$ and $D\in \huaD_n$, and  use \eqref{e-p1.1} to get the corresponding lower bounds on $\dim(\mu)$. Whilst for $7\leq n\leq 14$, we use Algorithm \ref{alg-3.8} (in which we take  $B=[0,1]$, $L=40$ and $\gamma=12.2*u$) to estimate $\mu(S_i^{-1}D)$ from above for  each $i$ and $D\in \huaD_n$, and use \eqref{e-p1.2} to get the corresponding lower bounds. In Table \ref{tab:Example4.1} we list our computational results.

 Finally, we make a brief error analysis about the computations involved in Table \ref{tab:Example4.1}. The rounding errors come from the computations using \eqref{e-p1.1} or \eqref{e-p1.2}. According to the
floating-point arithmetics, the major rounding errors come from the summation in \eqref{e-p1.1} or \eqref{e-p1.2}
over $D\in \huaD_n$, $n\leq 14$. By \cite[p.~82, (4.4)]{Higham2002}, the total error $E_n$ coming from  the summation over $D\in \huaD_n$ has a bound
\begin{align*}
|E_n|&\leq (\#\huaD_n-1)\times u \times \mbox{ (Total Absolute Sum)}+O(u^2)  \\
   &\leq (3^{14}-1)\times u\times \log 3 +0.01 u\\
   & \leq 5.8338\times 10^{-10}.
\end{align*}
So according to the numerical results in Table 3, we have $\dim(\mu)\geq 0.935825938$.}
\end{example}

\begin{table}
\begin{center}
\small
\begin{tabular}{c|c}
	Partition level $n$ & Lower bound on $\dim(\mu)$ \\ \hline
	         1          &      0.86058762883316      \\
	         2          &     0.873884695870383      \\
	         3          &     0.887965887736415      \\
	         4          &     0.901645728024083      \\
	         5          &     0.909541991955753      \\
	         6          &     0.915681937458243 	 \\
	         \hline
	         7          &     0.920399986771506      \\
	         8          &     0.924018201523078      \\
	         9          &     0.926957262754457      \\
	        10          &     0.929374519513162      \\
	        11          &     0.931389937165221      \\
	        12          &     0.933105444767198      \\
	        13          &     0.934566254269004      \\
	        14          &     0.935825938794224      \\ \hline
\end{tabular}
\\[2ex]
\caption{Lower bounds on the dimension of $\mu$ in Example \ref{exam-4.1}}
\label{tab:Example4.1}
\end{center}
\end{table}

\begin{example}
\label{exam-4.2}
Let $ \mu=\mu_{\beta_3}$ be the Bernoulli convolution with parameter $\beta_3$, where $\beta_3\approx 1.83928675521416$ is the tribonacci number. For $n\geq 1$, let $\huaD_n$ be the partition of $[0,1]$ generated by the endpoints of the intervals $S_{i_1\ldots i_n}([0,1])$, $i_1\ldots i_n\in \{1,2\}^n$, where $S_1(x)=x/\beta_3$, $S_2(x)=x/\beta_3+1-1/\beta_3$. Similar to Example \ref{exam-4.1}, we can use Theorem \ref{thm-1.0} (in which we take  $\huaD=\huaD_n)$ to obtain the lower bounds on $\dim \mu_{\beta_3}$ by either computing the precise values of $\mu(S_i^{-1}D)$, or by estimating $\mu(S_i^{-1}D)$ from above via Algorithm \ref{alg-3.8} (in which we take  $B=[0,1]$,  $L=40$ and $\gamma=10.2*u=10.2\times 2^{-53}$). Keep in mind that in this case, the finite sequence $(\Lambda_n)_{n=0}^{L^*}$ in Algorithm \ref{alg-3.8} is defined recursively by
\begin{equation*}
 \Lambda_{n+1}^*=\bigcup_{(t,[a,b])\in \Lambda_n} \left\{ \left(t/2,[S_i^{-1}(a)-\gamma, S_i^{-1}(b)+\gamma]\cap [0,1]\right), \;i=1,2\right\},
\end{equation*}
where $\beta:=\beta_3$, $S_1^{-1}(x)=\beta x$ and $S_2^{-1}(x)=\beta x+1-\beta$.

 In Table \ref{tab:Example4.2}, we list our computational results, where the values in the second column are obtained by using the first approach, and that in the third column  are obtained by using the second approach.  Again the major rounding errors  come from the summation in \eqref{e-p1.1} or \eqref{e-p1.2}
over $D\in \huaD_n$, $n\leq 14$, with a bound given by
\begin{align*}
|E_n|&\leq  (2^{14}-1)\times u\times \log 2 +0.01 u
    \leq 1.2608\times 10^{-12}.
\end{align*}

{\rm
\begin{table}
\begin{center}
\small
 \begin{tabular}{c|c|c}
	Partition level $n$ & Lower bound (I) & Lower bound (II) \\ \hline
	         1          &   0.974971672609929  &   0.974971672566547     \\
	         2          &   0.974971672609929  &   0.974971672568187     \\
	         3          &   0.974971672609929  &   0.974971672570255      \\
	         4          &   0.979950375568122  &   0.979950375495215      \\
	         5          &   0.979950375568122  &   0.979950375438316      \\
	         6          &   0.979950375568122  &   0.979950375332795 	 \\
	         7          &   0.980368793386354  &   0.980368792874582      \\
	         8          &   0.980368793386354  &   0.980368792346810	      \\
	         9          &   0.980368793386354  &   0.980368791273832      \\
	        10          &   0.980405622363758  &   0.980405618127927      \\
	        11          &   0.980405622363758  &   0.980405614052234      \\
	        12          &   0.980405622363758  &   0.980405606355333      \\
	        13          &   0.980408973316171  &   0.980408941535664      \\
	        14          &   0.980408973316170  &   0.980408909920482     \\ \hline
\end{tabular}
\\[2ex]
\caption{Lower bounds on the dimension of $\mu=\mu_{\beta_3}$ in Example \ref{exam-4.2}, using two different methods for the evaluation of  $\mu(S_i^{-1}D)$.}
\label{tab:Example4.2}
\end{center}
\end{table}
}

\end{example}

\section{A uniform lower bound on the dimension of Bernoulli convolutions}\label{S-5}
This section is concerned with a computer-assisted proof
of Theorem \ref{thm-1.1}.

  For $ \beta>1$, let $\mu_{\beta} $ be the Bernoulli convolution associated with $\beta$. This is, $\mu_{\beta}$ is the self-similar measure associated with the IFS
\[ \left \lbrace S_{1,\beta}(x)=\beta^{-1} x, \; S_{2,\beta}(x)=\beta^{-1} x+1-\beta^{-1} \right \rbrace \]
and the probability vector $ (\erfenzi,\erfenzi)$.

Let us begin with an elementary  result.

\begin{lem}
\label{lem-4.1}
\begin{itemize}
\item[(i)] For each $\beta>1$ and $k\in \N$,  $\dimmubeta\geq \dimmu{\beta^{k}}$. Consequently  if  $\beta^k\geq 2$, then 	$$\dimmubeta\geq \frac{\log 2}{k \log\beta}.$$
\item[(ii)] For $\beta \in [\sqrt{2},\; 1.424041]$,
$$\dimmubeta\geq 0.98041>\dim(\mu_{\beta_3}).$$
\end{itemize}
\end{lem}
\begin{proof}
Part(i) was proved in \cite[Proposition 2.1]{HareSidorov2018} for algebraic
parameter values $\beta$. The extension to the general parameters is similar in spirit. For completeness, we include a proof.  Let $\nu_\beta$ denote the probability distribution of the random series
\begin{equation}
\label{e-4.1}
\sum_{n=0}^\infty\epsilon_n \beta^{-n},
\end{equation}
 where $(\epsilon_n)$ is a sequence of independent and identically distributed random variables, taking the values $0$ and $1$ with equal probability. It is easy to check that $\mu_\beta(\cdot)=\nu_\beta\left(\frac{\beta}{\beta-1}\cdot\right)$, so $\dim(\mu_\beta)=\dim(\nu_\beta)$. Meanwhile, we note that for $k\in \N$,
$$
\nu_\beta=\nu_{\beta^k}*\eta
$$
for some probability measure $\eta$. To see this decomposition, consider the series \eqref{e-4.1} and separate the terms divisible by $k$ from the rest. Hence by \cite[Lemma 2.2]{FNW2002},
 $$
\dim(\nu_\beta)=\dim (\nu_{\beta^k}*\eta) \geq \dim(\nu_{\beta^k}).$$
  So $\dim(\mu_\beta)\geq \dim(\mu_{\beta^k})$. Whenever $\beta^k\geq 2$, the IFS $\{S_{1,\beta^k}, S_{2, \beta^k}\}$ satisfies the open set condition, it follows that $\dim(\mu_{\beta^k})=\log 2/(k \log \beta)$. This proves (i).

To see (ii), let $\beta \in [\sqrt{2},\; 1.424041]$. Then $\beta^2\geq 2$, so by (i),
$$
\dimmubeta\geq \frac{\log 2}{2 \log\beta}\geq \frac{\log 2}{2 \log 1.424041}\approx  0.980410065731842>\dim(\mu_{\beta_3}),
$$
where in the last inequality we used \eqref{e-1.0'}.
\end{proof}

Next we present our main method for  producing a uniform lower bound on  $\dim(\mu_\beta)$ when $\beta$ runs over  $(\sqrt{2},2)$.  For given $\beta>1$ and $N\in \N$, let $ \huaD_{N,\beta} $ be the partition of $ \unitinterval $ generated by the points in the following set
\begin{equation}
\label{e-s1.1}
\bigcup_{I\in \{1,2\}^N} \{S_{I, \beta}(0), \;S_{I, \beta}(1)\},
\end{equation}
where $S_{I,\beta}:=S_{i_1,\beta}\circ\cdots\circ S_{i_N,\beta}$ for $I=i_1\cdots i_N$. Let $f=f_2: \R^2_+\to \R$ be defined as in \eqref{e-f1.1} (in which we take $\ell=2$).

\begin{proposition}
\label{pro-4.3}
Let $\beta\in [\sqrt{2},2)$ and $\delta>0$.  Set
\begin{equation}
\label{e-4.2}
\epsilon:=\epsilon(\beta,\delta)=\left\{
\begin{array}{ll}
\frac{\delta}{\beta}(1+\frac{3}{\beta^4}) & \quad\mbox{ if  } \beta\leq 1.5,\\
\frac{\delta}{\beta}(1+\frac{2}{\beta^3}) & \quad \mbox{ if  } \beta> 1.5.
\end{array}
\right.
\end{equation}
For $N\in \N$, set
$$
t(\beta, \delta, N)=\sum_{[a,b]\in  \huaD_{N,\beta} }f\left(\frac{1}{2} \mu_\beta(S_{1,\beta}^{-1}[a-\epsilon,b+\epsilon]),\;  \frac{1}{2} \mu_\beta(S_{2,\beta}^{-1}[a-\epsilon,b+\epsilon])\right).
$$
Then for any $\beta'\in [\beta, \beta+\delta]$ with $\beta'\leq 2$,
\begin{equation}
\label{e-pro}
\dim(\mu_{\beta'})\geq \frac{(\log 2)-t(\beta, \delta, N)}{\log (\beta+\delta)}.
\end{equation}
\end{proposition}

To prove the above proposition, we need several lemmas.

\begin{lem}
\label{lem-4.4}
Let ${\bf i}=(i_n)_{n=1}^\infty\in \{0,1\}^\N$. Define $g_{\bf i}\colon (0,1)\to \Bbb R$ by
\begin{equation}
\label{e-4.3}g_{\bf i}(x):=(1-x)\sum_{n=1}^\infty i_n x^{n-1},\qquad x\in (0, 1).
\end{equation}
Then for any positive integer $k\geq 2$ and $x\in \left[1-\frac{1}{k}, 1-\frac{1}{k+1}\right)$,
$$
|g_{\bf i}'(x)|\leq kx^{k-1}.
$$
\end{lem}
\begin{proof}
Let $k\geq2$ and $x\in \left[1-\frac{1}{k}, 1-\frac{1}{k+1}\right)$. Then
$$
g_{\bf i}'(x)=-i_1+\sum_{n=2}^\infty i_n ((n-1)x^{n-2}-nx^{n-1})=I_1+I_2,
$$
where
\begin{align*}
I_1&:=-i_1+\sum_{n=2}^k i_n ((n-1)x^{n-2}-nx^{n-1}), \\
I_2&:=\sum_{n=k+1}^\infty i_n ((n-1)x^{n-2}-nx^{n-1}).
\end{align*}
Clearly, $I_1\leq 0$ and $I_2>0$. Moreover,
\begin{equation*}
\begin{split}
-I_1&=i_1+\sum_{n=2}^k i_n (nx^{n-1}-(n-1)x^{n-2})\\
&\leq 1+\sum_{n=2}^k  (nx^{n-1}-(n-1)x^{n-2})=kx^{k-1},
\end{split}
\end{equation*}
and
$$
I_2\leq \sum_{n=k+1}^\infty  ((n-1)x^{n-2}-nx^{n-1})=kx^{k-1}.
$$
Hence $|g_{\bf i}(x)|=|I_1+I_2|\leq \max\{-I_1, I_2\}\leq kx^{k-1}$.
\end{proof}

For $\beta>1$, let $ \pi_{\beta} \colon \lbrace1,2\rbrace^{\bfN}\to \R $ be the coding map associated with the IFS $\{S_{1,\beta}, \; S_{2,\beta}\}$. A direct calculation yields that
\begin{equation}\label{e-4.4}
\pi_\beta(x)=(1-\beta^{-1}) \sum_{n=1}^\infty (x_n-1) \beta^{-(n-1)},\qquad x=(x_n)_{n=1}^\infty.
\end{equation}
\begin{lem}
\label{lem-5.4}
Let $\sqrt{2}\leq \beta<\beta'\leq 2$.  Then
\begin{itemize}
\item[(i)] For any $u\in \lbrace1,2\rbrace^{\bfN}$,
$$
|\pi_\beta(u)-\pi_{\beta'}(u)|\leq \left\{
\begin{array}{ll}
2\beta^{-3} (\beta'-\beta)  &\quad \mbox{ if } \beta> 1.5,\\
3\beta^{-4} (\beta'-\beta)&\quad \mbox{ if } \beta\leq 1.5.
\end{array}
\right.
$$
\item[(ii)] For $c,d\in \R$ with  $c<d$,
$$
\mu_{\beta'}([c,\;d])\leq \left\{
\begin{array}{ll}
\mu_{\beta}([c-2\beta^{-3}(\beta'-\beta) ,\; d+ 2\beta^{-3}(\beta'-\beta)] )&\quad \mbox{ if } \beta> 1.5,\\
\mu_{\beta}([c-3\beta^{-4}(\beta'-\beta) , \; d+ 3\beta^{-4}(\beta'-\beta)] )&\quad \mbox{ if } \beta\leq 1.5.
\end{array}
\right.
$$

\end{itemize}
\end{lem}
\begin{proof}
We first prove (i). Let $u=(u_n)_{n=1}^\infty\in \lbrace1,2\rbrace^{\bfN}$.  Define ${\bf i}=(i_n)_{n=1}^\infty$ by $i_n=u_n-1$. By \eqref{e-4.3}--\eqref{e-4.4} and the mean value theorem,
\begin{equation}
\label{e-4.5}
|\pi_\beta(u)-\pi_{\beta'}(u)|=|g_{\bf i}(\beta^{-1})-g_{\bf i}((\beta')^{-1})|=\left(\frac{1}{\beta}-\frac{1}{\beta'}\right) |g_{\bf i}'(x)|\leq \beta^{-2}(\beta'-\beta) |g_{\bf i}'(x)|
\end{equation}
for some $x\in [1/\beta', 1/\beta]\subset [1/2, 1/\beta]$.

If $\beta>1.5$, then $x\in [1/2, 2/3)$ and by Lemma \ref{lem-4.4},  $|g_{\bf i}'(x)|\leq 2x\leq 2/\beta$, so by \eqref{e-4.5},
$|\pi_\beta(u)-\pi_{\beta'}(u)|\leq 2\beta^{-3}(\beta'-\beta)$.

Next assume that  $\sqrt{2}\leq \beta\leq 1.5$.  Since $x\in [1/2, 1/\beta]$,  either $x\in [1/2, 2/3)$ or $x\in [2/3, 3/4)$. If the first case occurs, then the argument in the last paragraph shows that  $|\pi_\beta(u)-\pi_{\beta'}(u)|\leq 2\beta^{-3}(\beta'-\beta)\leq 3\beta^{-4}(\beta'-\beta)$.  Else if $x\in [2/3, 3/4)$, then by Lemma  \ref{lem-4.4},
$|g_{\bf i}'(x)|\leq 3x^2\leq 3\beta^{-2}$, so by \eqref{e-4.5},
$|\pi_\beta(u)-\pi_{\beta'}(u)|\leq 3\beta^{-4}(\beta'-\beta)$.  This completes the proof of (i).

Next we prove (ii). Since $\mu_{\beta'}=m\circ \pi_{\beta'}^{-1}$ and $\mu_{\beta}=m\circ \pi_\beta^{-1}$, to prove (ii) it suffices to show that
$$
\pi_{\beta'}^{-1}([c,d])\subset \left\{
\begin{array}{ll}
\pi_{\beta}^{-1}([c-2\beta^{-3}(\beta'-\beta) ,\; d+ 2\beta^{-3}(\beta'-\beta)] )&\quad \mbox{ if } \beta> 1.5,\\
\pi_{\beta}^{-1}([c-3\beta^{-4}(\beta'-\beta) , \; d+ 3\beta^{-4}(\beta'-\beta)] )&\quad \mbox{ if } \beta\leq 1.5.
\end{array}
\right.
$$
Clearly the above inclusion follows from (i).
\end{proof}

Now we are ready to prove Proposition \ref{pro-4.3}.

\begin{proof}[Proof of Proposition \ref{pro-4.3}] Let $N\in \N$ and $\beta'\in [\beta,\beta+\delta]$ with $\beta'\leq 2$.
Since $\mu_{\beta'}$ is supported on $[0,1]$ and has no atoms, it follows that $\huaD_{\beta, N}$ is a finite Borel partition of $\R$ with respect to $\mu_{\beta'}$.  Applying Theorem \ref{thm-1.1}(i) to the IFS $\{S_{1,\beta'},\; S_{2,\beta'}\}$ and the probability vector $(1/2, 1/2)$,
$$
\dim(\mu_{\beta'})\geq \frac{\log 2 -\sum_{[a,b]\in \huaD_{N, \beta}}{f\left(\frac{1}{2}\mu_{\beta'}(S_{1,\beta'}^{-1}[a,b]),\; \frac{1}{2}\mu_{\beta'}(S_{2,\beta'}^{-1}[a,b])\right) }}{\log \beta'}
$$
By the above inequality and the increasing monotonicity of $f$,  to prove \eqref{e-pro} it suffices to show that
 for any $[a,b]\in \huaD_{\beta, N}$,
\begin{equation}
\label{e-4.6}
\mu_{\beta'}(S_{i,\beta'}^{-1}[a,b])\leq \mu_{\beta}(S_{i,\beta}^{-1}[a-\epsilon,b+\epsilon]),\quad i=1,2,
\end{equation}
where $\epsilon=\epsilon(\beta,\delta)$ is defined as in \eqref{e-4.2}. To  this end, set
$$
\xi=\left\{\begin{array}{ll}
2\beta^{-3}  &\quad \mbox{ if } \beta> 1.5,\\
3\beta^{-4} &\quad \mbox{ if } \beta\leq  1.5.
\end{array}
\right.
$$
Simply notice that
\begin{align*}
S_{1,\beta'}^{-1}[a,b]&=[\beta'a, \beta'b]\subset [\beta a, (\beta+\delta)b], \\
 \quad S_{2,\beta'}^{-1}[a,b]&=[\beta'a+1-\beta', \beta'b+1-\beta']\subset [(\beta+\delta)(a-1)+1, \beta(b-1)+1],
\end{align*}
and
\begin{align*}
S_{1,\beta}^{-1}[a-\epsilon,b+\epsilon]&=[\beta a-\beta \epsilon, \beta b+\beta \epsilon],\\
 S_{2,\beta}^{-1}[a-\epsilon,b+\epsilon]&=[\beta (a-\epsilon)+1-\beta, \beta (b+\epsilon)+1-\beta].
\end{align*}
Since $\beta\epsilon=\delta(1+\xi)$ and $[a,b]\subset[0,1]$,  it follows that
\begin{equation}
\begin{split}
[\beta a-\xi\delta, (\beta+\delta)b+\xi\delta]&\subset [\beta a-\beta \epsilon, \beta b+\beta \epsilon],\; \mbox{ and}\\
 [(\beta+\delta)(a-1)+1-\xi\delta, \beta(b-1)+1+\xi\delta]&\subset [\beta (a+\epsilon)+1-\beta, \beta (b+\epsilon)+1-\beta].
\end{split}
\end{equation}
Therefore for each $i\in \{1,2\}$, the  $\xi\delta$-neighborhood of $S_{i,\beta'}^{-1}[a,b]$ is contained in $S_{i,\beta}^{-1}[a-\epsilon,b+\epsilon]$.
Combining this fact with Lemma \ref{lem-5.4}(ii) yields \eqref{e-4.6}.
\end{proof}

Proposition \ref{pro-4.3} provides  a way to obtain a uniform lower bound on $\dim(\mu_\beta)$  when $\beta$ varies in a given interval.  For instance, let us take $\beta=1.42404$, $\delta=2\times 10^{-5}$ and $N=5$ in   Proposition \ref{pro-4.3}. A computation  using \eqref{e-pro}  shows that
$$
\dim(\mu_{\beta'})\geq 0.990857395851368 \qquad \mbox{ for all }\beta'\in [1.42404, 1.42406];
$$
in which we used Algorithm \ref{alg-3.8} (with  $B=[0,1]$, $L=28$, $\gamma=10.2\times 2^{-53}$)   to estimate $\mu_\beta(A)$ from above. Similarly, taking $\beta=1.42406$, $\delta=2\times 10^{-5}$ and $N=5$ in   Proposition \ref{pro-4.3} gives
$$
\dim(\mu_{\beta'})\geq 0.990863104536039 \qquad \mbox{ for all }\beta'\in [1.42406, 1.42408];
$$
In Tables \ref{tab:special1} and \ref{tab:special2}, we list our computational results for these (local) uniform lower bounds for those  $\beta$  near $ 1.42404$ or near $\beta_3$, respectively.

Now we are ready to prove Theorem \ref{thm-1.1}.

\begin{table}
\begin{center}
\small
\begin{tabular}{c|c|c|c|c|c}
$\beta$ &	Lower bound on $\dim(\mu_\beta)$  &	 $N$ & 	Iteration time 	& $\delta$	& Time Consumed\\
\hline
1.42404000 &	0.990857395851368 &	5 & 28 &   $ 2\times 10^{-5}$ &	46.2962464\\
1.42406000 &	0.990863104536039 &	5 &	28 &	$ 2\times 10^{-5}$ &	46.2679773\\
1.42408000 &	0.990865644424569 &	5 &	28 &	$ 2\times 10^{-5}$ &	45.9752501\\
1.42410000 &	0.990869982602642 &	5 &	28 &	$ 2\times 10^{-5}$ &	46.6695526\\
1.42412000 &	0.990871200514580 &	5 &	28 &	$ 2\times 10^{-5}$ &	45.4384355\\
1.42414000 &	0.990876033055317 &	5 &	28 &	$ 2\times 10^{-5}$ &	45.3021147\\
1.42416000 &	0.990877731084033 &	5 &	28 & 	$ 2\times 10^{-5}$ &	44.4827468\\
1.42418000 &	0.990883216572977 &	5 &	28 &	$ 2\times 10^{-5}$ &	44.6157998\\
1.42420000 &	0.990884812785932 &	5 &	28 &	$ 2\times 10^{-5}$ &	45.7283555\\
1.42422000 &	0.990887373414717 &	5 &	28 &	$ 2\times 10^{-5}$ &	46.1800343\\ \hline
\end{tabular}
\\[2ex]
\caption{Lower bounds on $\dim (\mu_\beta)$  when $\beta$ is near $1.42404$; the unit for time consumption is in seconds.}
\label{tab:special1}
\end{center}
\end{table}

\begin{table}
\begin{center}
\small
\begin{tabular}{c|c|c|c|c|c}
$\beta$ &	Lower bound on $\dim(\mu_\beta)$  &	 $N$ & 	Iteration time	& $\delta$	& Time Consumed \\
\hline
1.8392867549 &	0.980408601080113 &	13 &	40	& $10^{-10}$ &	7.8168833 \\
1.8392867550 &	0.980408591972940 &	13	& 40	& $10^{-10}$	& 8.261404 \\
1.8392867551 &	0.980408585976581 & 13	& 40	& $10^{-10}$	& 8.1949713 \\
1.8392867552 &	0.980408570326141 &	13	& 40	& $10^{-10}$	& 8.5678704\\
1.8392867553 &	0.980408569070358 &	13	& 40	& $10^{-10}$	& 8.3962353\\
1.8392867554 &	0.980408579672517 &	13	& 40	& $10^{-10}$	& 8.3294095\\
1.8392867555 &	0.980408593496653 &	13	& 40	& $10^{-10}$	& 8.524099\\
1.8392867556 &	0.980408601403820 &	13	& 40	& $10^{-10}$	& 9.1053152\\
1.8392867557 &	0.980408603444209 &	13	& 40	& $10^{-10}$	& 7.9313189\\
1.8392867558 &	0.980408612802920 &	13	& 40	& $10^{-10}$	& 8.6475025\\
\hline
\end{tabular}
\\[2ex]
\caption{Lower bounds on $\dim(\mu_\beta)$  for those $\beta$ near $\beta_3\approx 1.839286755214161$.}
\label{tab:special2}
\end{center}
\end{table}

\begin{proof}[Proof of Theorem \ref{thm-1.1}]
By Lemma \ref{lem-4.1}(i), it suffices to show that $
\dim(\mu_\beta)\geq 0.98040856
$
for all $\beta\in [\sqrt{2},2]$, and $\dim(\mu_\beta)>\dim(\mu_{\beta_3})$ if  $$\beta\in [\sqrt{2},2]\backslash (\beta_3-10^{-8},  \; \beta_3+10^{-8}).$$

Since $\dimmubeta\geq 0.98041$ for $\beta\in [\sqrt{2}, 1.42404]$ (see Lemma \ref{lem-4.1}(ii)), we only need to consider the parameters $\beta$ in the interval $[1.42404, 2]$.  To achieve our results, we further partition  this interval into 132530 tiny intervals and use the algorithm developed in Proposition \ref{pro-4.3} to compute the (local) uniform lower bound of $\dim(\mu_\beta)$ on each of these tiny intervals. In Table \ref{tab:UBparameters}, we give the precise information about our partition, as well as the choices of $N$, $\delta$, and the iteration time $L$ (used for the estimations of $\mu_\beta$ using Algorithm \ref{alg-3.8}, in which we take $B=[0,1]$ and $\gamma=10.2\times 2^{-53}$) for each of these tiny intervals.  For instance, the second line in Table \ref{tab:UBparameters} means that we partition the interval $[1.42404, 1.44]$ into sub-intervals of length $2\times 10^{-5}$, and for each such subinterval we apply Proposition \ref{pro-4.3} to calculate the corresponding lower bound in which we take $N=5$, $\delta=2\times 10^{-5}$ and 28 as the iteration time.

The full computational result on the (local) uniform lower bounds associated to  these 132530 intervals  is available at \href{https://github.com/zfengg/DimEstimate}{https://github.com/zfengg/DimEstimate}. A graphic illustration of this result is given in Figure \ref{tab:graph}.

According to this computational result, the smallest lower bound that we obtained is $0.980408569070358$, which is the uniform lower bound associated to  the subinterval $$[1.8392867553, 1.8392867554];$$
moreover, $\dim (\mu_\beta)>0.9804094>\dim(\mu_{\beta_3})$ for all $$
\beta\in [1.42404, 2]\backslash [1.8392867490,   1.8392867616].$$
That is enough to conclude Theorem \ref{thm-1.1}.
\end{proof}

\begin{table}[htbp]
	\centering
	\begin{tabular}{ll|c|c|c}
		BetaStart & BetaEnd    & $N$ & Iteration times & BetaStep $\delta$ \\ \hline
		1.42404   &   1.43998    &       5        &      28       &  2E-05   \\
		   1.44     &   1.45998    &       5        &      28       &  2E-05   \\
		   1.46     &   1.49998    &       5        &      28       &  2E-05   \\
		    1.5     &   1.68999    &       5        &      30       &  1E-05   \\
		   1.69     &   1.77999    &       6        &      30       &  1E-05   \\
		   1.78     &   1.799999   &       7        &      40       &  1E-06   \\
		    1.8     &   1.839199   &       7        &      40       &  1E-06   \\
		  1.8392    &  1.8392599   &       7        &      40       &  1E-07   \\
		  1.83926   &  1.83927399  &       7        &      40       &  1E-08   \\
		 1.839274   &  1.8392863   &       10       &      40       &  1E-08   \\
		1.83928631  & 1.839286579  &       10       &      40       &  1E-09   \\
		1.83928658  & 1.8392869339 &       13       &      40       &  1E-10   \\
		1.839286934 & 1.839287249  &       10       &      40       &  1E-09   \\
		1.83928725  &  1.83929899  &       10       &      40       &  1E-08   \\
		 1.839299   &  1.83930999  &       7        &      40       &  1E-08   \\
		  1.83931   &  1.8399999   &       7        &      40       &  1E-07   \\
		   1.84     &   1.849999   &       5        &      30       &  1E-06   \\
		   1.85     &   1.99999    &       5        &      30       &  1E-05\\
\hline
	\end{tabular}
\\[2ex]

	\caption{Partition of $[1.42404,2]$ and the corresponding $N$, $\delta$ and iteration times.}
	\label{tab:UBparameters}
\end{table}

\begin{figure}
 \begin{center}
\includegraphics[width=5in]{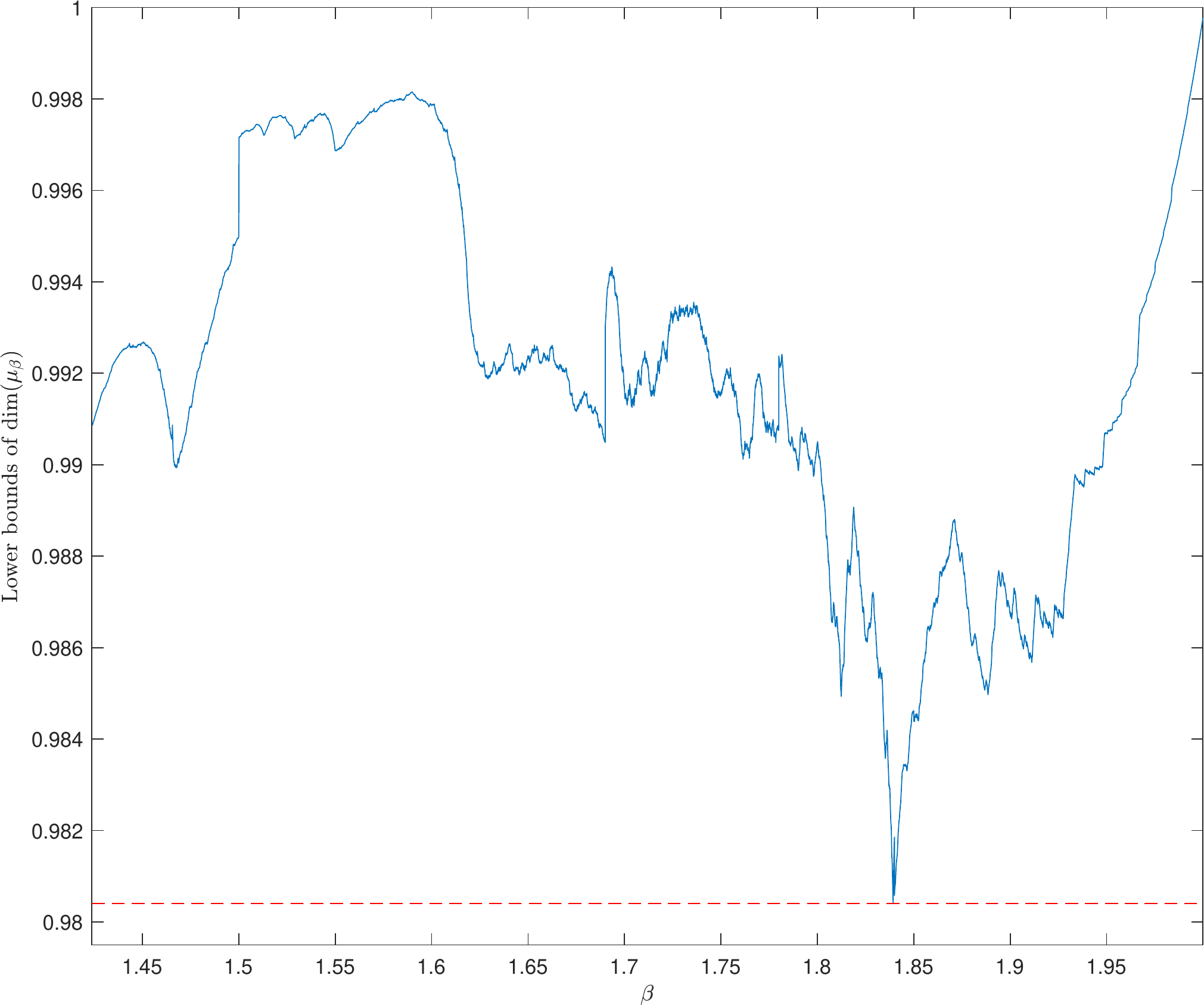}
\par\vspace{0pt}
\end{center}
\caption{A graphic illustration of our computational result on the (local) uniform lower bounds on $\dim(\mu_\beta)$. }
	\label{tab:graph}
\end{figure}

\section{Other theoretical results on Bernoulli convolutions}
\label{S-6}

For $ n=2,3,\ldots $, let $ \beta_{n} $ be the largest root of the polynomial $x^{n}-x^{n-1}-x^{n-2}-\cdots-1 $.
The first result of this section is the following.

\begin{proposition}
\label{pro-6.1}
\begin{itemize}
\item[(i)]For every $ \beta \in (1,2) $,
	 $\dimmubeta\geq \frac{\log
		2}{\log\beta}\cdot\mu_{\beta}\left([0,\beta-1]\right)$.
\item[(ii)]For each integer $n\geq 2$,
\begin{align}
\mu_{\beta_{n}}\left([0,\beta_{n}-1]\right)& =\frac{2^{n}-2}{2^{n}-1}\quad  \mbox{ and } \label{e-6.1}\\
\mu_{\beta}\left([0,\beta-1]\right) &\geq \frac{2^{n}-2}{2^{n}-1} \quad \mbox{ for  } \beta\in [\beta_n,\; 2). \label{e-6.2}
\end{align}
Consequently, $\displaystyle\dim(\mu_\beta)\geq \frac{2^{n}-2}{2^{n}-1}\cdot \frac{\log 2}{\log \beta_{n+1}}$ for $\beta\in [\beta_n,\beta_{n+1}]$, $n=2,3,\ldots.$
\end{itemize}
\end{proposition}
\begin{proof}
We first prove (i). Fix $\beta\in (1,2)$. Let $$\huaD=\lbrace[0,1-\beta^{-1}), [1-\beta^{-1},\beta^{-1}] , (\beta^{-1}, 1 ] \rbrace, $$ which is a finite Borel  partition of $ \unitinterval $. Applying Theorem \ref{thm-1.1}(i) to the IFS $$\{S_{1,\beta}(x)=\beta^{-1}x, \; S_{2,\beta}(x)=\beta^{-1}x+1-\beta^{-1}\}$$ and the probability weight $(1/2,1/2)$,
\begin{equation}
\label{e-6.3}
\dim(\mu_\beta)\geq \frac{\log 2 -\sum_{D\in \huaD}f\left(\frac{1}{2}\mu_\beta(S_{1,\beta}^{-1}D), \; \frac{1}{2}\mu_\beta(S_{2,\beta}^{-1}D)\right) }{\log \beta}.
\end{equation}
It is easily checked that $f(x,y)=0$ if one of $x$ and $y$ equals $0$, and $f(x,x)=2x\log 2$.  Meanwhile if $D=[0,1-\beta^{-1})$ or $(\beta^{-1}, 1 ]$, then one of $S_{1,\beta}^{-1}D$ and $S_{2,\beta}^{-1}D$ has no intersections with $[0,1]$, so has zero $\mu_\beta$ measure. It follows that
\begin{align*}
\sum_{D\in \huaD}&f\left(\frac{1}{2}\mu_\beta(S_{1,\beta}^{-1}D),\; \frac{1}{2}\mu_\beta(S_{2,\beta}^{-1}D)\right)\\
&=f\left(\frac{1}{2}\mu_\beta(S_{1,\beta}^{-1}[1-\beta^{-1},\beta^{-1}]), \; \frac{1}{2}\mu_\beta(S_{2,\beta}^{-1}[1-\beta^{-1},\beta^{-1}])\right)\\
&=f\left(\frac{1}{2}\mu_\beta([\beta-1,1]), \; \frac{1}{2}\mu_\beta([0, 2-\beta])\right)\\
&=\mu_\beta([\beta-1,1])\log 2,
\end{align*}
where in the last equality we use the property that $\mu_\beta([\beta-1,1])=\mu_\beta([0, 2-\beta])$, which follows from the symmetry of $\mu_\beta$ (i.e., $\mu_\beta([0,x])=\mu_\beta([1-x,1])$ for all $x\in [0,1]$).  Combining the above equality with \eqref{e-6.3} yields the desired inequality in (i).

Next we prove \eqref{e-6.1}. Fix an integer $n\geq 2$ and write $\rho:=\beta^{-1}_n$. We claim that
\begin{align}
\mu_{\beta_n}([0, 1-\rho])&=\frac{1}{2} \mu_{\beta_n}([0,\rho^{-1}-1])=\frac{1}{2} \mu_{\beta_n}([0,1-\rho^{n}]),\;\mbox{ and}\label{e-6.4}\\
\mu_{\beta_n}([0, 1-\rho^k])&=\frac{1}{2}+\frac{1}{2} \mu_{\beta_n}([0,1-\rho^{k-1}]),\quad k=2,\ldots, n. \label{e-6.5}
\end{align}
To see \eqref{e-6.4}, by the self-similarity of $\mu_{\beta_n}$,
$$
\mu_{\beta_n}([0, 1-\rho])=\frac{1}{2} \mu_{\beta_n}([0,\rho^{-1}-1])+\frac{1}{2} \mu_{\beta_n}([1-\rho^{-1}, 0])=\frac{1}{2} \mu_{\beta_n}([0,\rho^{-1}-1]).
$$
This proves \eqref{e-6.4}, since $\rho^{-1}-1=1-\rho^n$. Similarly for $2\leq k\leq n$,
\begin{align*}
\mu_{\beta_n}([0, 1-\rho^k])&=\frac{1}{2} \mu_{\beta_n}([0,\rho^{-1}-\rho^{k-1}])+\frac{1}{2} \mu_{\beta_n}([1-\rho^{-1},1-\rho^{k-1}])\\
&=\frac{1}{2} +\frac{1}{2} \mu_{\beta_n}([0,1-\rho^{k-1}]),
\end{align*}
where in the second equality we use the fact that $$\rho^{-1}-\rho^{k-1}=(1+\rho+\cdots+\rho^{n-1})-\rho^{k-1}\geq 1.$$
This proves \eqref{e-6.5}.
Solving the linear equations in \eqref{e-6.4}--\eqref{e-6.5} gives $$ \mu_{\beta_n}([0,\rho^{-1}-1])=\mu_{\beta_n}([0,1-\rho^n])=\frac{2^n-2}{2^n-1},$$
which proves \eqref{e-6.1}.

Finally we prove \eqref{e-6.2}.  Fix an integer $n\geq 2$ and $\beta\in [\beta_n,2)$.  We first show that
\begin{equation}
\label{e-6.7}
\beta (\beta^{n-1}-\beta^{n-2}-\cdots-1)\geq 1,
\end{equation}
and
\begin{equation}
\label{e-6.6}
\beta (\beta-1)(\beta^{i}-\beta^{i-1}-\cdots-1) \geq 1\; \mbox{ for }  i=1,\ldots, n-2, \mbox{ provided that }n\geq 3,
\end{equation}
To prove the above inequalities, notice that
\begin{equation*}
\beta\geq \beta_n=1+\beta_n^{-1}+\cdots+\beta_n^{-(n-1)}\geq 1+\beta^{-1}+\cdots+\beta^{-(n-1)}.
\end{equation*}
Hence
$$
\beta^{n-1}-\beta^{n-2}-\dots-1\geq \beta^{n-2}(1+\beta^{-1}+\cdots+\beta^{-(n-1)})-\beta^{n-2}-\cdots-1=\beta^{-1},
$$
from which  \eqref{e-6.7} follows. Moreover, for $i=1,\ldots, n-2$ (provided that $n\geq 3$),
\begin{align*}
\beta^i-\beta^{i-1}-\cdots-1&\geq \beta^{i-1}(1+\beta^{-1}+\cdots+\beta^{-(n-1)})-\beta^{i-1}-\cdots-1\\
&=\beta^{-1}+\cdots+\beta^{i-n}\\
&\geq \beta^{-1}+\beta^{-2},
\end{align*}
so $$\beta (\beta-1)(\beta^{i}-\beta^{i-1}-\cdots-1)\geq \beta (\beta-1)(\beta^{-1}+\beta^{-2})=\beta-\beta^{-1}\geq 1.$$
This proves \eqref{e-6.6}.

By the self-similarity of $\mu_\beta$,
\begin{equation}
\label{e-6.8}
\begin{split}
\mu_\beta([0, \beta -1])&= \frac{1}{2}\mu_\beta([0, \beta^2-\beta]) +\frac{1}{2}\mu_\beta([0, (\beta-1)(\beta-1)]) \\
&= \frac{1}{2} +\frac{1}{2}\mu_\beta([0, (\beta-1)(\beta-1)]),
\end{split}
\end{equation}
where in the second equality we use the fact that $\beta^2-\beta\geq 1$.
Similarly,
	\begin{equation}
	\label{e-6.9}
     	\begin{split}
     	\mu_\beta&([0, (\beta -1) (\beta^{n-1}-\beta^{n-2}-\cdots-1)])\\
	&\geq \frac{1}{2} \mu_\beta([0, \beta(\beta -1) (\beta^{n-1}-\beta^{n-2}-\cdots-1)]) \\
     	&\geq \frac{1}{2}\mu_\beta([0,\beta -1] )\qquad (\mbox{by \eqref{e-6.7}}).
     	\end{split}
     	\end{equation}	
Moreover,
for $i=1,\ldots, n-2$ (provided $n\geq 3$),
 \begin{equation}
 \label{e-6.10}
      	\begin{split}
	&\mu_\beta([0, (\beta -1) (\beta^i-\beta^{i-1}-\cdots-1)])\\
	&= \frac{1}{2}\mu_\beta([0, \beta(\beta -1) (\beta^i-\beta^{i-1}-\cdots-1)])+\frac{1}{2}\mu_\beta([0, (\beta -1) (\beta^{i+1}-\beta^{i}-\cdots-1)]) \\
&= \frac{1}{2} +\frac{1}{2}\mu_\beta([0, (\beta -1) (\beta^{i+1}-\beta^{i}-\cdots-1)])\qquad (\mbox{by \eqref{e-6.6}}).
     	\end{split}
     	\end{equation}

To complete the proof of \eqref{e-6.2}, we consider the cases $n=2$ and $n\geq 3$ separately.  First assume that $n=2$. Then by \eqref{e-6.8}--\eqref{e-6.9},
$$
\mu_\beta([0, \beta -1])=\frac{1}{2} +\frac{1}{2}\mu_\beta([0, (\beta-1)^2])\quad \mbox{ and } \quad
\mu_\beta ([0, (\beta -1)^2])\geq \frac{1}{2}\mu_\beta([0, \beta -1]),
$$
from which we deduce that $\mu_\beta([0, \beta -1])\geq 2/3$. This proves \eqref{e-6.2} in the case that $n=2$.

In what follows we assume $n\geq 3$.  Set $x=\mu_\beta([0, \beta-1])$ and
$$
y_i= \mu_\beta([0, (\beta -1) (\beta^i-\beta^{i-1}-\cdots-1)])\quad \mbox{ for }  i=1,\ldots, n-1.$$
By  \eqref{e-6.8} and \eqref{e-6.10},
$$
x=\frac{1}{2}+\frac{1}{2}y_1\quad \mbox{ and }\quad  y_i=\frac{1}{2}+\frac{1}{2}y_{i+1},\quad i=1,\ldots, n-2,
$$
from which we deduce that $y_{n-1}=2^{n-1}x -\left(2^{n-1}-1\right)$. Meanwhile by \eqref{e-6.9}, $y_{n-1}\geq x/2$. Hence
$2^{n-1}x -\left(2^{n-1}-1\right)\geq x/2$, from which we obtain
$$
\mu_\beta([0, \beta-1])=x\geq \frac{2^{n}-2}{2^{n}-1}.
$$
This completes the proof of \eqref{e-6.2}.
\end{proof}

\begin{remark}
\label{rem-6.2}
It is known  that $\beta_n=2-2^{-n}+O(\frac{n}{4^n})$ (see \cite[Lemma 3.3]{FengWang2004}). Applying this fact and Proposition \ref{pro-6.1}, it is easy to show that there exists  a constant $c>0$ such that
$$1-\dim \mu_\beta\leq c(2-\beta) \mbox{  for all }\beta\in (1,2).$$  This plays a complement to the following inequality obtained in \cite[Theorem 3] {KPV2021}:
$$1-\dim \mu_\beta\leq c(\beta-1) \mbox{ for some  } c>0 \mbox{ and all } \beta\in (1,2).$$
Meanwhile, according to the theoretic formula of $\dim (\mu_{\beta_n})$ (see \cite{Feng1999, Feng2005, Grabner2002}), one can check that
$$
\dim (\mu_{\beta_n})=1-\left(1-\frac{1}{\log 4}\right)2^{-n}+O\left(\frac{n}{4^n}\right),
$$
which implies that $$\lim_{n\to \infty} \frac{1-\dim (\mu_{\beta_n})}{2-\beta_n}=1-\frac{1}{\log 4}.$$
\end{remark}
\medskip

 In the remaining part of this section, we prove  the following.
 \begin{lem}
 \label{lem-6.3}
 There exists $\beta_{\ast} \in (\sqrt{2},2) $ such that
	\[ \dimmu{\beta_{\ast}} = \underset{\beta\in(1,2]}{\inf} \dimmubeta.\]
 \end{lem}
 \begin{proof}

For each $\beta\in (1,\sqrt{2})$, there exists a positive integer $k$ such that $\beta^k\in [\sqrt{2},2]$. By Lemma \ref{lem-4.1},  $\dim (\mu_\beta)\geq \dim(\mu_{\beta^k})$. It follows that $$\underset{\beta\in(1,2]}{\inf} \dimmubeta=\underset{\beta\in[\sqrt{2},2]}{\inf} \dimmubeta.$$

It is known that the mapping $\beta \mapsto \dim(\mu_\beta)$ is lower semi-continuous on $(1,2]$ (see e.g.~\cite[Theorem~1.8]{HochmanShmerkin2012}).  Hence there exists  $\beta_*\in [\sqrt{2},2]$ so that
$$
\dimmu{\beta_{\ast}} = \underset{\beta\in [\sqrt{2},2]}{\inf} \dimmubeta.
$$
Notice that $\dim(\mu_\beta)=1$ if $\beta=2$. Combining  this with the inequality $\dim(\mu_\beta)\geq \dim (\mu_{\beta^2})$ (cf. Lemma \ref{lem-4.1}(i)) yields that
 $\dim(\mu_{\sqrt{2}})=1$. Therefore $\beta_*\not\in \{\sqrt{2},2\}$. This proves the lemma.
\end{proof}

\section{Upper bound estimate for the dimension of Bernoulli convolutions associated with Pisot numbers}\label{S-7}

In this section we prove Theorem \ref{thm-1.4} and give our computational results on the upper and lower bounds on $\dim(\mu_\beta)$ for some examples of Pisot numbers $\beta$ of degree $3$ or $4$.

We begin with the proof of Theorem \ref{thm-1.4}.
\begin{proof}[Proof of Theorem \ref{thm-1.4}]

Let $P(q)$, $q>0$, be defined as in \eqref{e-P}. It is proved in \cite[Theorems 3.3--3.4]{FengLau2002} that $P$ is differentiable on $(0,\infty)$ and for each $q>0$, there is a unique $\sigma$-invariant measure $\eta_q$ on $\Sigma$ satisfying the following Gibbs property
$$
\eta_q([i_1\cdots i_n])\approx e^{-nP(q)} \|A_{i_1}\cdots A_{i_n}\|^q\quad \mbox{  for }n\in \N \mbox{ and }i_1\cdots i_n\in \{1,\ldots, k\}^n,
$$
where $\|\cdot\|$ stands for the standard matrix norm.  Furthermore by \cite[Theorems 1.1, 1.2 and 3.1]{Feng2004}, $P$ satisfies the following variational relation
$$
P(q)=h_{\eta_q}(\sigma)+q\lim_{n\to \infty}\frac1n\int \log \|A_{x|n}\| \; d\eta_q(x),\qquad q>0
$$
and the derivative formula
$$
P'(q)=\lim_{n\to \infty}\frac1n\int \log \|A_{x|n}\| \; d\eta_q(x),\qquad q>0,
$$
where $A_{x|n}:=A_{x_1}\cdots A_{x_n}$ for $x=(x_n)_{n=1}^\infty\in \Sigma$.
Combining the above two equalities yields that $P(q)=h_{\eta_q}(\sigma)+qP'(q)$. Taking $q=1$ and using the fact that $P(1)=0$ we obtain $P'(1)=-h_{\eta_1}(\sigma)$. By the definition of $\eta$ (see \eqref{e-G}), $\eta$ satisfies the Gibbs property
$\eta([i_1\cdots i_n])\approx  \|A_{i_1}\cdots A_{i_n}\|$, so $\eta=\eta_1$. Thus  $P'(1)=-h_{\eta}(\sigma)$.  Combining it with \eqref{e-dim} yields $\dim(\mu_\beta)=h_{\eta}(\sigma)/(\log \beta)$.

To complete our proof, let $ \calP =\Braced{[i]:\; i=1,\ldots, k} $ be the partition of $\Sigma:=\{1,\ldots,k\}^\N $ consisting of the first order cylinders in $\Sigma$. Then
	\begin{equation}\nonumber
	\begin{split}
	u_{n} &:= \sum_{I\in \{1,\ldots,k\}^{n+1}}\varphi\left(\eta([I]) \right) -
	\sum_{J\in \{1,\ldots,k\}^{n}}\varphi\left(\eta([J]) \right)\\
	&= H_{\eta}\left(\calP\vee \cdots \vee \sigma^{-n} \calP\right) -
	H_{\eta}\left(\calP\vee \cdots \vee \sigma^{-(n-1)} \calP\right)\\
	&= H_{\eta}\left(\calP\vee \cdots \vee \sigma^{-n} \calP\right) -
	H_{\eta}\left(\sigma^{-1}\calP\vee \cdots \vee \sigma^{-n} \calP\right)\\
	&= H_{\eta}\left(\calP\vert\left(\sigma^{-1}\calP\vee \cdots \vee \sigma^{-n} \calP\right)\right),
	\end{split}
	\end{equation}
where we used \cite[Theorem 4.3(ii)]{Walters1982} in the last equality.  By \cite[Theorem 4.3(iii)]{Walters1982}, the sequence $(u_{n})$ is decreasing.  Now by definition,
	\begin{equation}\nonumber
	\begin{split}
	h_{\eta}(\sigma) & =\nliminf \frac{1}{n} H_{\eta}\left(\calP\vee \cdots \vee \sigma^{-(n-1)} \calP\right)\\
&= \nliminf \frac{1}{n} \sum_{\abs{I}=n} \varphi\left(\eta([I])\right)\\
	&= \nliminf  \frac{1}{n} (u_{n} + u_{n-1} + \cdots + u_{1})
	\end{split}
	\end{equation}
		Since $ (u_{n}) $ is decreasing, it follows that $h_{\eta}(\sigma) =\nliminf u_{n} = \underset{n}{\inf} \: u_{n}$.
\end{proof}

In the remaining part of this section  we give some computational results on the upper and lower bounds on $\dim(\mu_\beta)$ for some examples of Pisot numbers $\beta$ of degree $3$ or $4$.

We begin with the illustration of the computations for the example in which $\beta\approx 1.465571232$ is the largest root of the polynomial
$x^3-x^2-1=0$.  We use the inequality \eqref{e-p1.1} in Theorem \ref{thm-1.0}  to give a finite sequence of low bounds on $\dim(\mu_\beta)$, where $\huaD=\huaD_{N,\beta}$ is the partition of $[0,1]$ generated by the points in the set given in \eqref{e-s1.1}, whilst we use the inequality \eqref{e-1.7} in Theorem 1.4 to give a finite sequence of upper bounds on $\dim(\mu_\beta)$; in this example, there are $46$ constructed matrices $A_i$'s of dimension $d=346$.    In Table \ref{tab:special} we list these lower and upper bounds.

\begin{table}
\begin{center}
\tiny
\begin{tabular}{c|c|c|c|c|c}
Level $N$	& Low bound	& Time consumed	& Level	$n$ & Upper bound	& Time consumed\\
\hline
	10 &	0.999446403056440 &	0.054791699	& 6	& 0.999553838975762 & 0.466657771	\\
	11 &	0.999485559701407 &	0.019805691	& 7	& 0.999551745301729 & 0.077114112	\\
	12 &	0.999511565006248 &	0.018707125	& 8	& 0.999549324828679 & 0.098623572	\\
	13 &    0.999524092865024 &	0.031473356	& 9 &	0.999548108265720 &	0.099398655	\\
    14 &	0.999533483697397 &	0.027698097	& 10 &	0.999546974504613 &	0.122857504	\\
	15 &	0.999537083624124 &	0.035850279	& 11 &	0.999546321331208 & 0.138546403	\\
	16 &	0.999539945814714 &	0.054405927	& 12 &	0.999545837403620 &	0.198051248	\\
	17 &	0.999541706067639 &	0.063491041	& 13 &	0.999545467625853 &	0.247629358	\\
	18 &	0.999542557004638 &	0.071829385	& 14 &	0.999545258823762 &	0.312467417	\\
	19 &	0.999543459170884 &	1.607735407	& 15 &	0.999545072687759 &	0.584398307	\\
	20 &	0.999543735484899 &	0.186936405	& 16 &	0.999544971532786 &	0.618315478	\\
	21 &	0.999544108747012 &	0.208533991	& 17 &	0.999544888243190 &	1.015030255	\\
	22 &	0.999544301645519 &	0.243305524	& 18 &	0.999544831886938 &	1.195914636	\\
	23 &	0.999544402854973 &	0.355187983	& 19 &	0.999544797950427 &	1.611427142	\\
	24 &	0.999544527192754 &	0.563398845	& 20 &	0.999544767910898 &	2.332147467	\\
	25 &	0.999544565385905 &	3.173328553	& 21 &	0.999544751961036 &	3.621352815	\\
	26 &	0.999544619835200 &	1.163999234	& 22 &	0.999544738388558 &	5.046995913	\\
	27 &	0.999544645262069 &	1.653751317	& 23 &	0.999544729649307 &	8.976078225	\\
	28 &	0.999544663526862 &	2.458708739	& 24 &	0.999544723974222 &	10.85274788	\\
	29 &	0.999544681239203 &	3.52355423	& 25 &	0.999544719277577 &	15.76524517	\\
	30 &	0.999544687025939 &	5.013234467	& 26 &	0.999544716737165 &	22.61266586	\\
	31 &	0.999544695999151 &	7.36229044	& 27 &	0.999544714509664 &	33.20685747	\\
	32 &	0.999544699587091 &	10.73880172	& 28 &	0.999544713164084 &	48.45996338	\\
	33 &	0.999544702739358 &	15.77259785	& 29 &	0.999544712224285 &	71.26915581	\\
	34 &	0.999544705355331 &	23.01813448	& 30 &	0.999544711488671 &	104.4124192	\\
	35 &	0.999544706362411 &	34.01411608	& 31 &	0.999544711078597 &	152.6973032\\ \hline 	
\end{tabular}
\\[2ex]
\caption{Lower and upper  bounds on $\dim (\mu_\beta)$  where $\beta\approx 1.465571232$ is the largest root of
$x^3-x^2-1$; the unit for time consumption is in seconds.}
\label{tab:special}
\end{center}
\end{table}

In a similar way we  compute the lower and upper bounds on $\dim(\mu_\beta)$ for 4 other Pisot numbers of degree 3 or 4.  In Table \ref{tab:degree5}, we list the corresponding computational results briefly.

\begin{table}
\begin{center}
\small
\begin{tabular}{l|c|c|c}
	          Polynomial &   $\beta$   & lower bound on $\dim(\mu_\beta)$ & upper bound on $\dim(\mu_\beta)$ \\ \hline
	    $x^3 - x^2 -  1$ & 1.465571232 &        0.999544706362411         &        0.999544711078597         \\
	    $x^3 - x    - 1$ & 1.324717957 &        0.999995036655607         &        0.999995037372877         \\
	      $x^3-2x^2+x-1$ & 1.754877666 &        0.994020046394375         &        0.994020065372927         \\
	 $x^4- 2x^3 + x - 1$ & 1.866760399 &        0.991391363780141         &        0.991401387766015         \\
	$x^4- x^3 -2x^2 + 1$ & 1.905166167 &        0.989449155226028         &        0.989601151164740         \\ \hline
\end{tabular}
\\[2ex]
\caption{Lower and upper  bounds on $\dim (\mu_\beta)$  for some Pisot numbers  $\beta$
  of degree 3 or 4. }
\label{tab:degree5}
\end{center}
\end{table}

\section{Final remarks}
\label{S-8}
In this section, we give several final remarks.

\begin{remark} One could obtain further sharper uniform lower bounds on the dimension of Bernoulli convolutions if he/she partitions the interval $[\beta_3- 10^{-8}, \beta_3+ 10^{-8}]$ into  subintervals of length much smaller than $10^{-10}$ and manages to compute the (local) unform lower bounds of $\dim \mu_\beta$ associated to these subintervals (in which taking $N>13$ and an iteration time $L>40$).
\end{remark}

\begin{remark} The method of estimating projection entropies can be also used to find lower bounds for the dimension of certain self-affine measures. To be more precise, let $\mu$ be the self-affine measure associated with an affine IFS $\{S_i(x)=A_ix+a_i\}_{i=1}^\ell$ on $\R^d$ and a probability vector $(p_1,\ldots, p_\ell)$, where $A_i$ are contracting invertible $d\times d$ matrices.  Suppose that these matrices are all diagonal, then $\dim(\mu)$ can be expressed as the linear combinations of projection entropies associated with several coding maps (see \cite[Theorems~2.11--2.12]{FengHu2009}), hence we can provide lower bounds for $\dim(\mu)$ by estimating these projection entropies. Below we give a concrete example.
\end{remark}

\begin{example}
\label{example-8.1}
Let $\mu$ be the self-affine measure associated with an affine IFS $\{S_1, S_2\}$ on $\R^2$ and the probability vector $(1/2, 1/2)$, where
$$
 S_1(x,y)=\left(\frac{x}{\alpha},\; \frac{y}{\beta}\right), \quad S_2(x,y)=\left(\frac{x}{\alpha}+1-\frac{1}{\alpha},\; \frac{y}{\beta}+1-\frac{1}{\beta}\right),
$$
and $\alpha,\beta$ are two parameters with $1<\alpha<\beta<2$.
Let $m=\prod_{n=1}^\infty\{1/2,1/2\}$, and let $\pi$, $\pi_1$ denote the canonical coding maps associated with the IFSs $\{S_1, S_2\}$ and $\{\alpha^{-1}x, \alpha^{-1}x+1-\alpha^{-1}\}$,  respectively. Then by \cite[Theorem~2.11]{FengHu2009},
\begin{align*}
\dim(\mu) &=\left(\frac{1}{\log \alpha}-\frac{1}{\log \beta}\right)h_{\pi_1}(\sigma,m)+\frac{1}{\log \beta}h_\pi(\sigma,m)\\
&=\left(\frac{1}{\log \alpha}-\frac{1}{\log \beta}\right)\left(\log 2-H_m(\mathcal P|\pi_1^{-1}\mathcal B(\R)\right)+\frac{\log 2-H_m\left(\mathcal P|\pi^{-1}\mathcal B(\R^2)\right)}{\log \beta},
\end{align*}
where $\mathcal P=\{[1], [2]\}$ is the natural partition of $\{1,2\}^\N$. For each given pair $(\alpha,\beta)$, we can  apply \eqref{e-U} to numerically estimate the conditional entropies $H_m(\mathcal P|\pi_1^{-1}\mathcal B(\R))$ and $H_m(\mathcal P|\pi^{-1}\mathcal B(\R^2))$ from above, which will lead to a lower bound on $\dim \mu$. In Table \ref{tab:affine}, we
present our computations results for some chosen parameters $(\alpha,\beta)$.

{\rm
\begin{table}
\begin{center}
\small
\begin{tabular}{c|c|c|c}
	          $(\alpha, \beta)$ & Upper bound on    & Upper bound on & Lower bound on $\dim(\mu)$  \\
& $H_m(\mathcal P|\pi_1^{-1}\mathcal B(\R))$ & $H_m(\mathcal P|\pi^{-1}\mathcal B(\R^2))$ & \\
 \hline
 $(1.2, 1.4)$ & 0.594519457635335  & 0.381337271355742   &    1.17453512577913         \\
$(1.3, 1.7)$ & 0.443677110679849   & 0.011983272405036 & 1.76440615371483 \\
$(1.4, 1.8)$ & 0.358042443198116   & 0.000020606397791 & 1.60503743978513 \\
	    $(1.5, 1.6)$               & 0.287856436058623  &   0.010339956295331   &    1.59002600229069         \\
        $(1.5, 1.7)$               & 0.287856436058623  &   0.000000000000000   &    1.54205227015933         \\
        $(1.7, 1.71)$              & 0.165391636766638  &   0.105399412288277   &    1.10640907763501   \\
        $(1.8, 1.9)$               & 0.064653016798699  &   0.035399973460099   &    1.03287878328719        \\
 \hline
\end{tabular}
\\[2ex]
\caption{Lower bounds  on the dimension of  $\mu$ in Example \ref{example-8.1}.}
\label{tab:affine}
\end{center}
\end{table}
}

\end{example}

\medskip
\medskip

{\noindent \bf  Acknowledgements}.  This research was partially supported by the General Research Fund CUHK14304119 from the Hong Kong Research Grant Council and  a direct grant for
research from the Chinese University of Hong Kong. The authors would like to thank the referee for helpful comments and suggestions.


\begin{thebibliography}{10}

\bibitem{AFKP2020}
Shigeki Akiyama, De-Jun Feng, Tom Kempton, and Tomas Persson.
\newblock On the {H}ausdorff dimension of {B}ernoulli convolutions.
\newblock {\em Int. Math. Res. Not. IMRN}, (19):6569--6595, 2020.

\bibitem{AlexanderZagier1991}
J.~C. Alexander and Don Zagier.
\newblock The entropy of a certain infinitely convolved {B}ernoulli measure.
\newblock {\em J. London Math. Soc. (2)}, 44(1):121--134, 1991.

\bibitem{Bowen1975}
Rufus Bowen.
\newblock {\em Equilibrium states and the ergodic theory of {A}nosov
  diffeomorphisms}.
\newblock Lecture Notes in Mathematics, Vol. 470. Springer-Verlag, Berlin-New
  York, 1975.

\bibitem{BV2020}
Emmanuel Breuillard and P\'{e}ter~P. Varj\'{u}.
\newblock Entropy of {B}ernoulli convolutions and uniform exponential growth
  for linear groups.
\newblock {\em J. Anal. Math.}, 140(2):443--481, 2020.

\bibitem{Feng1999}
De-Jun Feng.
\newblock On the limit {R}ademacher functions and {B}ernoulli convolutions.
\newblock In {\em Dynamical Systems: {P}roceedings of the {I}nternational
  {C}onference in {H}onor of {P}rofessor {L}iao {S}hantao},  ed. by Yunping Jiang and Lan Wen. World
  Scientific, Singapore, pages 46--49,  1999.

\bibitem{Feng2003}
De-Jun Feng.
\newblock Smoothness of the {$L^q$}-spectrum of self-similar measures with
  overlaps.
\newblock {\em J. London Math. Soc. (2)}, 68(1):102--118, 2003.

\bibitem{Feng2004}
De-Jun Feng.
\newblock The variational principle for products of non-negative matrices.
\newblock {\em Nonlinearity}, 17(2):447--457, 2004.

\bibitem{Feng2005}
De-Jun Feng.
\newblock The limited {R}ademacher functions and {B}ernoulli convolutions
  associated with {P}isot numbers.
\newblock {\em Adv. Math.}, 195(1):24--101, 2005.

\bibitem{FengHu2009}
De-Jun Feng and Huyi Hu.
\newblock Dimension theory of iterated function systems.
\newblock {\em Comm. Pure Appl. Math.}, 62(11):1435--1500, 2009.

\bibitem{FengLau2002}
De-Jun Feng and Ka-Sing Lau.
\newblock The pressure function for products of non-negative matrices.
\newblock {\em Math. Res. Lett.}, 9(2-3):363--378, 2002.

\bibitem{FNW2002}
De-Jun Feng, Nhu~T. Nguyen, and Tonghui Wang.
\newblock Convolutions of equicontractive self-similar measures on the line.
\newblock {\em Illinois J. Math.}, 46(4):1339--1351, 2002.

\bibitem{FengWang2004}
De-Jun Feng and Yang  Wang.
\newblock Bernoulli convolutions associated with certain non-Pisot numbers.
 \newblock {\em Adv. Math.}, 187:173--194, 2004.


\bibitem{Garsia1963}
Adriano~M. Garsia.
\newblock Entropy and singularity of infinite convolutions.
\newblock {\em Pacific J. Math.}, 13:1159--1169, 1963.

\bibitem{Grabner2002}
Peter~J. Grabner, Peter Kirschenhofer, and Robert~F. Tichy.
\newblock Combinatorial and arithmetical properties of linear numeration
  systems.
\newblock {\em Combinatorica}, 22(2):245--267, 2002.

\bibitem{HKPS2019}
Kevin~G. Hare, Tom Kempton, Tomas Persson, and Nikita Sidorov.
\newblock Computing {G}arsia entropy for {B}ernoulli convolutions with
  algebraic parameters.
\newblock {\em preprint, arXiv:1912.10987}, 2019.

\bibitem{HareSidorov2018}
Kevin~G. Hare and Nikita Sidorov.
\newblock A lower bound for the dimension of {B}ernoulli convolutions.
\newblock {\em Exp. Math.}, 27(4):414--418, 2018.

\bibitem{Higham2002}
Nicholas~J. Higham.
\newblock {\em Accuracy and stability of numerical algorithms}.
\newblock Second edition. Society for Industrial and Applied Mathematics (SIAM),
              Philadelphia, PA, 2002.


\bibitem{Hochman2014}
Michael Hochman.
\newblock On self-similar sets with overlaps and inverse theorems for entropy.
\newblock {\em Ann. of Math. (2)}, 180(2):773--822, 2014.

\bibitem{Hochman2015}
Michael Hochman.
\newblock On self-similar sets with overlaps and inverse theorems for entropy
  in {$\Bbb R^d$}.
\newblock {\em preprint, arXiv:1503.09043}, 2015.
\newblock To appear in {\it Mem. Amer. Math. Soc.}

\bibitem{Hochman2018}
Michael Hochman.
\newblock Dimension theory of self-similar sets and measures.
\newblock In {\em Proceedings of the {I}nternational {C}ongress of
  {M}athematicians---{R}io de {J}aneiro 2018. {V}ol. {III}. {I}nvited
  lectures}, pages 1949--1972. World Sci. Publ., Hackensack, NJ, 2018.

\bibitem{HochmanShmerkin2012}
Michael Hochman and Pablo Shmerkin.
\newblock Local entropy averages and projections of fractal measures.
\newblock {\em Ann. of Math. (2)}, 175(3):1001--1059, 2012.

\bibitem{HornJohnson1985}
Roger~A. Horn and Charles~R. Johnson.
\newblock {\em Matrix analysis}.
\newblock Cambridge University Press, Cambridge, 1985.

\bibitem{Hutchinson1981}
John~E. Hutchinson.
\newblock Fractals and self-similarity.
\newblock {\em Indiana Univ. Math. J.}, 30(5):713--747, 1981.

\bibitem{KPV2021}
Victor Kleptsyn, Mark Pollicott, and Polina Vytnova.
\newblock Uniform lower bounds on the dimension of {B}ernoulli convolutions.
\newblock {\em preprint, arXiv:2102.07714}, 2021.

\bibitem{Lalley1998}
Steven~P. Lalley.
\newblock Random series in powers of algebraic integers: {H}ausdorff dimension
  of the limit distribution.
\newblock {\em J. London Math. Soc. (2)}, 57(3):629--654, 1998.

\bibitem{PSS2000}
Yuval Peres, Wilhelm Schlag, and Boris Solomyak.
\newblock Sixty years of {B}ernoulli convolutions.
\newblock In {\em Fractal geometry and stochastics, {II}
  ({G}reifswald/{K}oserow, 1998)}, volume~46 of {\em Progr. Probab.}, pages
  39--65. Birkh\"{a}user, Basel, 2000.

\bibitem{Rapaport2020a}
Ariel Rapaport.
\newblock Proof of the exact overlaps conjecture for systems with algebraic
  contractions.
\newblock {\em preprint, arXiv:2001.01332}, 2020.
\newblock To appear in {\it Ann. Sci. \'{E}c. Norm. Sup\'{e}r. (4)}.

\bibitem{Rapaport2020b}
Ariel Rapaport and P\'{e}ter~P. Varj\'{u}.
\newblock Self-similar measures associated to a homogeneous system of three
  maps.
\newblock {\em preprint, arXiv:2010.01022}, 2020.

\bibitem{Varju2018}
P\'{e}ter~P. Varj\'{u}.
\newblock Recent progress on {B}ernoulli convolutions.
\newblock In {\em European {C}ongress of {M}athematics}, pages 847--867. Eur.
  Math. Soc., Z\"{u}rich, 2018.

\bibitem{Varju2019}
P\'{e}ter~P. Varj\'{u}.
\newblock On the dimension of {B}ernoulli convolutions for all transcendental
  parameters.
\newblock {\em Ann. of Math. (2)}, 189(3):1001--1011, 2019.

\bibitem{Walters1982}
Peter Walters.
\newblock {\em An introduction to ergodic theory}, volume~79 of {\em Graduate
  Texts in Mathematics}.
\newblock Springer-Verlag, New York-Berlin, 1982.

\end{thebibliography}
\end{document}